\documentclass[11pt]{article}
\usepackage{amsfonts,amsmath,amstext,amssymb,bm}
\usepackage{amsthm}
\usepackage{hyperref} 
\usepackage{array}
\usepackage{longtable}
\setlongtables

\newtheorem{thm}{Theorem}
\newtheorem{lem}{Lemma}

\newtheorem{prop}{Proposition}
\newtheorem{rem}{Remark}

\newtheorem{algorithm}{Algorithm}
\def\R{\mathbb{R}}

\def\N{\mathbb{N}}
\def\Z{\mathbb{Z}}
\def\P{\mathbb{P}}
\newcommand{\set}[1]{\left\{{#1}\right\}}
\newcommand{\eset}[2]{\left\{{#1} : \: {#2}\right\}}
\newcommand{\indic}[1]{1_{\{#1\}}}
\newcommand{\Nat}{\mathbb{N}}
\renewcommand{\Re}{\mathbb{R}}
\newcommand{\D}{{\,\mathrm{d}}}

\usepackage[pdftex]{graphicx}
\usepackage[caption=false,font=footnotesize]{subfig}
\usepackage{epstopdf}
\usepackage{color}

\usepackage{ifthen}

\newcommand{\E}[1]{\mathbb{E}#1}
\newcommand{\Epo}[1]{{\mathbb{E}^p}#1}
\newcommand{\PP}[1]{\mathbb{P}#1}
\newcommand{\pr}[1]{\mathbb{P}\hspace{-0.1em}\left(#1\right)}
\newcommand{\cPr}[2]{\mathbb{P}\hspace{-0.6mm}\left[\left.#1\,\right|#2\right]}
\newcommand{\cE}[2]{\mathbb{E}\left[\left.#1\,\right|#2\right]}

\allowdisplaybreaks

\begin{document}

\setlength{\pdfpagewidth}{8.5in}
\setlength{\pdfpageheight}{11in}

\title{On Spatial Point Processes\\ with Uniform Births\\ and Deaths by
Random Connection}

\author{
Fran{\c c}ois Baccelli,\hspace{1cm} 
Fabien Mathieu, \hspace{1cm} Ilkka Norros\\
UT Austin, USA, 
\quad Bell Laboratories, France,\quad
VTT, Finland }

\maketitle

\begin{abstract}
This paper is focused on a class of spatial birth and death process
of the Euclidean space where the birth rate is constant and the death
rate of a given point is the shot noise created at its location by
the other points of the current configuration for some 
response function $f$. An equivalent view point is that
each pair of points of the configuration establishes a random
connection at an exponential time determined by $f$, which results
in the death of one of the two points. We concentrate on space-motion
invariant processes of this type. Under some natural conditions on $f$,
we construct the unique time-stationary regime of this class of
point processes by a coupling argument. We then use the birth and death
structure to establish a hierarchy of balance integral relations between
the factorial moment measures. Finally, we 
show that the time-stationary point
process exhibits a certain kind of repulsion between its points
that we call $f$-repulsion.
\end{abstract}
\section{Introduction}

This paper introduces a class of motion
invariant \footnote{By motion invariant, we mean
spatial birth and death processes
invariant by all translations and rotations of the
Euclidean space.}
on the Euclidean space where births take place according to
some homogeneous Poisson rain and where the instantaneous
death rate of a point of the current configuration
is the shot-noise \cite{Stoyan:1995:SGA} of the configuration at this point 
for some positive response function $f$.

The analysis of these dynamics was initially
motivated by models arising in peer to peer networking \cite{BMNV13}.
This class of processes is however of more general potential interest 
as it features non-trivial interactions between points combining ``density'' and ``geometry'' components
that should have other practical incarnations. 
Here, this comes from the fact that the death rate cannot be assessed through densities only as connections are 
also functions of distances. In fact, the presence of a point at some location
implies that there are less points around than what a density argument
would predict. This Palm type bias
\cite{DalVJon:88} makes the role of geometry quite central.

The main mathematical results on this class of processes are 
(i) an existence and uniqueness result on their stationary regimes
(Theorem \ref{thm:wholephi});
(ii) a hierarchy of balance equations linking their
factorial moment measures \cite{DalVJon:88}
of neighboring orders (Theorem \ref{thm:baleq});
(iii) a general repulsion result formalizing the
Palm bias alluded to above (Theorem \ref{thm:repulsivity}). 

These results can be seen as a complement to those of Garcia and Kurtz in \cite{garciakurtz06}. 
In this last paper, the authors also considered a spatial birth and death process on the whole Euclidean space
but in the case where the birth rate (rather than the death rate here) depends on the configuration.

The first sections of the paper are focused on the construction
of the stationary regime
of such spatial birth and death processes.
There is no fundamental difficulty in building such a stationary regime
when the phase space is compact
using the formalism of Preston \cite{preston75}
together with Markovian techniques.
The main challenge
addressed here is hence that of the construction when the 
phase space is the whole Euclidean space.
A pathwise construction of this steady state is proposed.
The first step of this approach consists in building
the state for all compacts of time and a large enough class
of initial conditions. This construction 
leverages the random connection interpretation of the
shot noise death process. It consists 
in a recursive investigation for determining which connection
is responsible of each individual death.
It is defined in Section \ref{sec:exist} and  is called {\em Sheriff}.
The second step builds a coupling between the
dynamics with an empty initial condition
and that with a motion invariant initial
condition ${\cal Z}_0$. This construction, which is
defined in Section \ref{sec:init} is again pathwise and a natural
extension of Sheriff which is called
Sheriff$^Z$.  We then show that, in the coupling of Sheriff$^Z$,
the influence of each point of ${\cal Z}_0$ almost surely dies out
in finite time. This is based on martingale and random walk arguments
which are given at the end of Section \ref{sec:init}. 
Tightness and positiveness arguments (Section \ref{sec:globtight}) are then combined with differential equations on the densities and on second moment measures (Section \ref{sec:compl}) to prove that on any compact of space, the time of last influence of ${\cal Z}_0$ is actually integrable.
This allows one to develop a coupling from the past argument (Section \ref{sec:ergo})
which proves the existence and uniqueness result.

Section \ref{sec:mom} gathers a few basic properties on the stationary distribution in question. We show that the differential equations alluded to above lead to a set of conservation
laws for moment measures that mimic the Markov birth and death
structure: the $k$-th moment measure is balanced by certain
integral forms of the $k-1$-st and the $k+1$-st, with the
usual reflection at $k=0$. The stationary regime is also
shown to exhibit $f$-repulsion, a property which translates
the fact that a typical point of the stationary configuration
suffers of a smaller death rate than that seen by the typical
locus of the Euclidean space.

Lastly, In the appendix, (Section \ref{sec:appendix}), we detail the proofs of some of the Equations and provide a table of notation.

The model and its dual representation in term of either shot-noise or random connections is described in the following Section. 

\section{Model} \label{sec:model}

We start with two informal definitions of the stochastic process
of interest. We then give a formal definition of the problem.

\subsection{Model Description}

\subsubsection{Spatial Birth and Death Viewpoint}
Let $D$ be a closed convex set of $\R^d$.
Let $M(D)$ denote the set of counting measures $\phi$ on $D$
(see e.g. \cite{Ka76}).
Depending on the situation, we will consider the point process
$\phi$ either as a counting measure, or as a set, the support of this counting measure.  As a result, the number of points of $\phi$
in the Borel set $C$ will be denoted either $\phi(C)$
or $|\phi\cap C|$, with $|S|$ the cardinality
of the set $S$, depending on the circumstances.
Let $\cal M(D)$ denote the smallest $\sigma$-field containing all
the events $\phi(C)=k$, $C$ ranging over Borel subsets of $D$
and $k$ over integers.

We consider a {\em {spatial birth and death}} (SBD) process
on $D$, namely a $M(D)$-valued Markov jump process \cite{preston75}.
The state (or point configuration) of this Markov process at time $t$
will be denoted by $\Phi_t\in M(D)$.

It is well known \cite{preston75} that when $D$ is compact,
such a Markov process is characterized
by two rate functionals, the birth rate functional $b(\phi, \phi+\delta_x)$,
which gives the infinitesimal rate of a birth at $x\in D$ in
configuration $\phi\in M(D)$ and the death rate
functional $\mu(\phi+\delta_x,\phi)$,
which gives the infinitesimal rate of the death of $x\in \phi+\delta_x$ in
configuration $\phi+\delta_x\in M(D)$.

The birth rate functional considered in the present paper 
is homogeneous in time and space, namely
\begin{equation}
b(\phi, \phi+\delta_x) =\lambda ,
\end{equation}
for all $x\in D$ and $\phi\in M(D)$, where $\lambda$ is a positive
real number.

Let $f:\R^+\to \R^+$ be a non-negative function which we will call
the {\em response function} of the model. The death rate 
of the SBD process considered in the present paper
is determined by this function through the relation
\begin{equation}
\label{eq:death}
\mu(\phi+\delta_x,\phi)=\sum_{y\in\phi}f(||x-y||).
\end{equation}
It is homogeneous in time but not in space:
the death rate of $x$ in configuration $\phi+\delta_x$
is the {\em shot noise} created by $\phi$ at $x$ for the
response function $f$.

In the compact domain case
the finite time horizon problem can be analyzed by classical
Markov chain uniformization techniques and
the existence and the uniqueness of the time stationary
regimes can be proved using
the theory of Markov chains in general state spaces \cite{MT93}.

The object of interest in this paper is the extension of these dynamics to $\R^d$.
When $D=\R^d$, the above Markov approach fails
even for the construction of the finite time horizon state.

\subsubsection{Death by Random Connection Viewpoint}
\label{sec:rcg}
Another equivalent description of the dynamics is in terms
of a Random Connection Graph (RCG).
A RCG \cite{pen91,FM07} on a point process $\Phi\in M(D)$
is informally defined as follows: for all $\Phi\in M(D)$,
for all unordered pairs $\{x,y\}$ of points of $\Phi$,
one samples an independent
Bernoulli random variable $Q(x,y)$ with value $1$ with 
probability $c(||x-y||)$
and $0$ with probability $1-c(||x-y||)$.
The function $c: \R^+\to [0,1)$ will be referred to as the connection function.
The associated random connection
model is the random graph on $\phi$ with edges between
the points $(x,y)$ such that $Q(x,y)=1$. 

Informally, the SBD process studied in this paper
can also be obtained by sampling, for all
unordered pairs $\{x,y\}$ of points of 
$\Phi$, an independent exponential random
variable $T_{xy}$ with rate $2f(||x-y||)$ and in establishing
at this time a lethal connection between $x$ and $y$ which instantly kills
either of the two with probability 1/2, independently of everything else.
This death can, however, only happen if the points $x$ and $y$ are still
alive at time $T_{xy}$, which is not guaranteed as each might
have already been killed by other points.
It should be clear that, at least in the case
where $D$ is compact and the time interval is compact as well, the death rate
of any given point $x\in \phi$ is then given by (\ref{eq:death})
as the deaths that occur in state $\phi$ in an infinitesimal interval of time 
with length $dt$ can be obtained by sampling with probability 1/2
the points connected by edges in a RCG on 
$\phi$ with connection function $c(r)=2f(r) dt$. 
In view of this, it makes sense to call this mechanism
{\em death by random connection}.

This second view point will be instrumental for constructing 
the process on $\R^d$.
\subsection{Problem Statement}
\label{sec:rcg+}
We will represent the births as
a $\R^d\times \R$ Poisson point process, $\Psi$, i.e.
the births in the time interval $(t_0,t_1)$ are
$\Psi_{(t_0,t_1)}=\Psi_{\R^d\times (t_0,t_1)}$, a Poisson process on
$\Re^d\times(t_0,t_1)$ with intensity measure $\lambda l^d\times
l(t_0,t_1)$, where $ l^d $ (resp. $ l $) stands for the Lebesgue measure of $ \R^d $ (resp. $ \R $).  A point $p\in\Psi$ will also be denoted by $(x_p,b_p)$,
with $x_p\in \R^d$ the location of the birth and $b_p\in (t_0,t_1)$
the time of the birth. The point process $\Psi_{\R^d\times (t_0,\infty)}$
will be denoted by $\Psi_{t_0}$.

For any two points $p,q\in\Psi$, let $I_{pq}$ and
$T_{pq}$ be two random variables independent of everything else with
distributions
\begin{eqnarray*}
I_{pq}&=& 1-I_{qp}\sim\mbox{ Bernoulli}(\frac12),\\
T_{pq}&=&T_{qp}\sim(b_p\vee b_q)+\mbox{Exp}(2f(\|x_p-x_q\|)).
\end{eqnarray*}
These quantities have the interpretations alluded to above:
\begin{itemize}
\item $T_{pq}$ is the
time at which the connection between $p$ and $q$ is realized (it
will actually be the death of one of them, if both are alive just before
$T_{pq}$);
\item $I_{pq}=1$ if the direction of the
connection is from $q$ to $p$ (the dying point will be $p$ if
both are alive just before $T_{pq}$; $q$ is then said to kill
$p$ at time $T_{pq}$).
\end{itemize}
As long as both points are alive, they both ``feel''
a (time) intensity $f(\|x_p-x_q\|)$ to be killed by the other. 

The dynamics of interest can be defined by the equation:
\begin{equation}\label{eq:recurs}
d_p=\inf\eset{T_{pq}}{q\in\Psi,\ d_q\ge T_{pq},\ I_{pq}=1}.
\end{equation}
Note that the condition $d_q\ge T_{pq}$ makes the definition
recursive with respect to time.
Here is a continuous time version of the last equation in terms
of a stochastic differential equation: for all bounded sets $C$
of $\R^d$,
\begin{equation}\label{eq:recurs-cont}
\D{} \Phi_t(C)
= \Psi(C,\D{t}) - \sum_{X\in \Phi_t} \sum_{Y\ne X\in \Phi_t} 
\delta_X(C)
N(X,Y,dt),
\end{equation}
where $t\to \Psi(C,t)$ is a Poisson point process of intensity
$\lambda |C|$ on $\R$ and $t\to N(x,y,t)$, $x,y\in \R^d$, is a
collection of independent Poisson point processes with $N(x,y,t)$ of
intensity $f(||x-y||)$ on $\R$ for all $x,y\in \R^d$. Note that in
each realization only a countable number of these Poisson processes
comes to the scene.

The general problem can be stated in the following terms: given some
initial condition $\Phi_0$, which is some point process in $\R^d$, (i)
can one construct a solution $\{\Phi_t\}$ to (\ref{eq:recurs-cont})
where $\Phi_t$ is a point process on $\R^d$ for all $t>0$?
(ii) if so, under what conditions does $\Phi_t$
converge in distribution to a limit? (iii) does this limit, when
it exists, depend on the initial condition?

\subsection{Assumptions on the Response Function}
Throughout the paper, when $D=\R^d$, the following properties on $f$ will
be considered:
\begin{itemize}
\item {\bf Assumption 0}: $f$ is non negative and $f(0)=0$.
\item {\bf Assumption 1}
\begin{equation}
\label{eq:afinite}
0<a<\infty\text{, where }a:=\int_{\Re^d}f(\|x\|)dx \text{.}
\end{equation}
\item {\bf Assumption 2}:
the function $r\to f(r)$ is monotone non-increasing on $(0,\infty)$.
\item {\bf Assumption 3}:
the function $f$ is bounded above. We will then denote by $K$ the
upper-bound on $f$.
\end{itemize}
Assumption 0 is natural in this context; the assumption that
$f(0)=0$ makes sense as we always deal with simple point processes.
Assumption 1 is used throughout the paper. 
This assumption is used for 
proving that events can be sorted out in $\R^d$
(Lemma \ref{nonaccondprop} and Theorem \ref{sheriffworksprop}).
Assumption 2--3 are only needed 
in the final steps of the construction of the stationary
regime; they are not required for the construction on 
compacts of time.
If Assumption 2 holds (which we do not assume in general), 
the death rate is higher in regions with many points.

\section{Construction on Finite Time Horizon} \label{sec:exist}

The main question addressed in this section is whether there 
exists a solution to (\ref{eq:recurs-cont}) 
(or equivalently to (\ref{eq:recurs})). 
We use the connection--death view point described in
Sections \ref{sec:rcg} and \ref{sec:rcg+}
to construct these dynamics pathwise over all compacts of time and space.

It should first be noticed that (\ref{eq:recurs})
may be problematic if the set 
$$
N_p=\eset{T_{pq}}{q\in\Psi_{t_0}}
$$ 
has accumulation points. The following proposition gives a
condition guaranteeing that this is a.s.\ not the case when $t_0$ is finite.

\begin{lem}
\label{nonaccondprop}
Assume that $t_0>-\infty$ and that Assumptions 0--1 hold. Then
almost surely none of the sets $N_p$, $p\in\Psi_{t_0}$, has
accumulation points. 
\end{lem}
\begin{proof}
For any $p=(x,t)\in \R^d\times [t_0,\infty)$, the conditional distribution of 
$\Psi_{t_0}-\delta_p$
given that $\Psi_{t_0}$ has a point in $p$ is
a Poisson process with same distribution as
$\Psi_{t_0}$ (Slivnyak's theorem).
In the following, $\Epo{\ }$ denotes this conditional
expectation, or equivalently the Palm distribution
of $\Psi_{t_0}$ at $p$.

To each point $q=(y,s)$ of $\Psi_{t_0}$, we associate the 
point $T_{pq}$ of $\R$.

We show that the intensity measure 
of this point process on $\R$ is locally finite under the
condition given above.
For any $u\ge t$,
\begin{eqnarray*}
\Epo{|N_p\cap(t_0,u]|}
&=&\Epo{\int_{\Re^d\times(t_0,\infty)}\indic{T_{pq}\le u}
(\Psi_{t_0}-\delta_p)(\D{q})}\\
&=&\E{\int_{\Re^d\times(t_0,u]}\indic{T_{pq}\le u}
  (\Psi_{t_0})(\D{q})}\\
&=&\lambda\int_{\Re^d}\int_{t_0}^u\,
  \pr{\mathrm{Exp}(f(\|x-y\|))\le u-(t\vee v)}\D{v} \D{y}\\
&\le&\lambda(u-t_0)\int_{\Re^d}\left(1-e^{-(u-t_0)f(\|y\|)}\right)\D{y}\\
&\le&\lambda(u-t_0)^2
  \int_{\Re^d}f(\|y\|)\D{y} =\lambda(u-t_0)^2 a <\infty.
\end{eqnarray*}
Here Exp($z$) denotes an exponential random variable of parameter $z$; 
the third equality is Campbell's formula;
we used the inequality $1-e^{-z}\le z$ in the last line.
\end{proof}

\begin{rem}
The last lemma holds under the weaker
assumption $a(1):= \int_{\Re^d\setminus B(0,1)}f(\|x\|)dx<\infty$.
\end{rem}

Thus, for all $p$, every finite interval of $[t,\infty)$ contains an a.s.\ 
finite number of points of the type $T_{pq}$, $q\in\Psi_{t_0}$.
Note that Lemma \ref{nonaccondprop} fails with $t_0=-\infty$.

\subsection{The Sheriff Algorithm}
To construct the death process when $t_0>-\infty$ and $t_1<\infty$,
we propose below an algorithm that we call the Sheriff algorithm.
Its name comes from the following Western
imagery: there is wild shooting in an infinite saloon of $\R^d$
with cowboys arriving over time and space; 
the sheriff has to find out
who is still alive at a given time and who was killed by whom before this
time.

Within the setting of Section \ref{sec:rcg+},
the general idea is quite natural: one picks a node, checks 
its earliest connection (potential death) time; in order to determine whether this
is its actual death time, one has to determine whether the
death time of the killer is earlier or later than this time (Equation (\ref{eq:recurs}));
for this, one checks the earliest connection time of the latter, etc. \\

\begin{algorithm}
\label{sheriff1}
\rm
\renewcommand{\labelenumi}{\arabic{enumi}} 
\renewcommand{\theenumi}{\arabic{enumi}} 
\quad\newline
\noindent
The Sheriff Algorithm:
Construction of the death process on time interval $(t_0,t_1)$.
\begin{enumerate}
\item  Initialization:
\begin{itemize}
\item Every point born in $(t_0,t_1)$, say $p=(x,t)$ with
$t_0\le t \le t_1$ has a stack of its {\em death sentences}.
A death sentence for $p$ is a triple $(p,q,T_{pq})$, where $q$
is a potential killer of $p$, i.e.\ $I_{pq}=1$.
Death sentences are sorted earliest on top\footnote{Note that, under
the assumptions of Lemma \ref{nonaccondprop},
each stack has simple sequential order.};
\item The sheriff has an investigation stack, initially empty.
\end{itemize}
\item\label{sheriffcycle} If the investigation stack is empty, the
  sheriff chooses, from a pre-defined ordering of all
  points\footnote{By this, we mean a bijection between the points of
    the configuration and $\N$; for instance, points can be sorted
    in function of their distance to the origin of the Euclidean
    space, and ties, if any, can be solved in a random way.}, the
  first point whose stack has on top a death sentence with time less
  than $t_1$ and no death certificate, and moves the sentence
  to the investigation stack. If
  there is no such point, then the procedure ends
\footnote{We shall see that, under our assumptions of an infinite
domain, the sheriff never stops.}. 
\item The sheriff looks at the sentence on top of the investigation stack,
say $(p,q,T)$, and does one of the following:
\begin{itemize}
\item If killer $q$'s stack has on top a death sentence or death certificate 
with a time larger than $T$, then $q$ is alive at $T$ and the execution
happens. 
The sheriff changes the sentence $(p,q,T)$ into a death certificate
with the same data and returns it to the top of $p$'s stack.
\item  If $q$ has a death certificate earlier than $T$, then
the execution is not realized and
the sentence $(p,q,T)$ is discarded, i.e.\ the investigation stack
is popped.
\item Otherwise the sheriff moves 
the top sentence of $q$'s stack to the investigation stack.
\end{itemize}
\item Go to \ref{sheriffcycle}.
\end{enumerate}
\end{algorithm}

\noindent
To prove that the Sheriff Algorithm works properly,
we start with the following lemma which
shows that for all finite intervals $(t_0,u)$ as above,
for all predefined ordering of the points, for all cards,
the recursive investigation performed by Sheriff to determine 
the status of this card ends in finite time.

\begin{lem}
\label{finiteinvestigationlemma}
Assume that $t_0>-\infty$ and that $f$
satisfies Assumptions 0--1. Then,
almost surely, there is no infinite sequence of points $p_1,p_2,\ldots$
such that the sequence $T_{p_np_{n+1}}$ is non-increasing.
\end{lem}
\begin{proof}
Notice that since $t_0>-\infty$, the sequence $T_{p_np_{n+1}}$ is
bounded from below.

The proof uses percolation properties of the Poisson RCG in $\R^d$
\cite{pen91,FM07}. Let us view the arrival locations $x_p$ of
the points $p$ of $\Psi_{t_0}$ arrived until time
$u$ as a homogeneous Poisson $\Phi$
point process of $\R^d$ with intensity $\lambda(u-t_0)$.
The time of arrival $b_p$ of point $p$ is seen as 
an independent mark, uniform on $(t_0,u)$.

Let $J$ be some time interval of $(t_0,u)$.
We create an undirected edge (a connection) between
the points $x_p$ and $x_q$ of $\Phi$ if $T_{pq}\in J$.
This does not form a RCG because
of the marks (in the RCG,
one establishes an edge between two points
of a Poisson point process with a probability that 
depends on their distance only; here the mark
of point $x_p$ creates a correlation between the
edges that connect $x_p$ to the other points).

Consider now the model where one creates an undirected edge between
$x_p$ and $x_q$ of $\Phi$ if $M_{x_px_q}(J)>0$ where
$M_{x_px_q}$ is a Poisson point process on $\R$ with intensity
$f(\|x_p-x_q\|)$, conditionally independent of
everything else given $\|x_p-x_q\|$ . By a standard coupling argument,
this defines a {\em dominating RCG}, i.e.\ a RCG where
there are more edges than in the original model.

The mean number of connections of point $x\in \Phi$
in this dominating RCG is
\begin{eqnarray*}
\E^x \left[ \sum_{y\ne x\in \Phi} 
\P (M_{xy}(J)>0 \mid \Psi) \right]
& \le & 
\E^x \left[ \sum_{y\ne x\in \Phi} 
\E (M_{xy}(J) \mid \Psi) \right]\\
& = &
\E^x \left[ \sum_{y\ne x\in \Phi}  
l(J) f(\|x-y\|) \right]\\
& = & \lambda l(J) \int_{\R^d} f(\|z\|) \D{z} = \lambda l(J) a.
\end{eqnarray*}
Here, $\E^x$ refers to the Palm probability of $\Phi$ at $x$;
the second bound uses the fact that the probability that
a non negative integer valued random variable is positive
is less than its mean; the last relation
follows from Slinyak's theorem and Campbell's formula.
Hence, if the Lebesgue measure $l(J)$ of $J$ is small enough,
there is hence no percolation in this dominating RCG
\cite{pen91,FM07}. As a result, there is no percolation in the initial model.

This last property immediately implies that for all $p_1$ with
$x_{p_1}=x$ and all non-increasing sequences
$T_{p_1,p_2},T_{p_2,p_3},\cdots$
with $T_{p_1,p_2}=t$, we have $T_{p_k,p_{k+1}}< t-\epsilon$
for all $k$ larger than some random but finite $K$.
This then proves the result of the lemma by a finite induction.
\end{proof}
 
An important property which remains to be proved is that
the result of the Sheriff algorithm does not depend
of the ordering of points that it relies upon.
This is the object of the following:

\begin{thm}
\label{sheriffworksprop}
Assume that $t_0>-\infty$ and that Assumptions 0--1 hold. Then
almost surely, for every point $p$ born in $(t_0,t_1)$, with $t_1<\infty$,
the Sheriff either determines a unique death time $d_p\le t_1$ and
the killer, or finds out that $p$ is alive at time $t_1$. 
The result is independent of the order in which the points were enumerated.
This uniquely defines the point process $\Phi_t$ of nodes
alive at time $t$ for all $t>0$.
\end{thm}
\begin{proof}
Consider Algorithm \ref{sheriff1}. Make Step 2, and consider the set
$P$ of all points whose stacks are looked at before the investigation
stack is emptied again. By Lemma \ref{nonaccondprop}, all stacks
contain, a.s., only a finite number of cards with $T_{pq}<t_1$, 
and all times $T_{pq}$ are a.s.\ distinct. If $P$ is
infinite, it contains a sequence with the property appearing in Lemma
\ref{finiteinvestigationlemma}. Thus, with probability one, the
investigation stack empties. Repeating the cycle, every point, sooner
or later and almost surely, either gets a death certificate or has all
connection times in its stack larger than $t_1$, in which case it is alive
at $t_1$.

It remains to show that, almost surely, the resulting configuration does
not depend on the pre-defined order in which the points are
investigated. Let us consider one realization of the triple
$(\Psi_{(t_0,t_1)},\set{T_{pq}},\set{I_{pq}})$ and two
different numberings of the points, say $\set{p^{(1)}}$ and
$\set{p^{(2)}}$. Since every point gets, a.s., a death certificate in
both processes, we only need to show that these certificates are
a.s.\ identical in both processes. Assume that for some point $p_1$ we
have $d^{(1)}_{p_1}>d^{(2)}_{p_1}$, where the superscripts refer to
the two orderings of points. There exists a $p_2$ such that
$d^{(2)}_{p_1}=T_{p_1p_2}$. Since the
card $(p_1,p_2,T_{p_1p_2})$ is present in process 1
and since $p_1$ is alive after time $T_{p_1p_2}$ in process 1, 
it must not be killed by $p_2$ at $T_{p_1p_2}$ in process 1
(keeping in mind that $d^{(2)}_{p_1}=T_{p_1p_2}$ implies
$I_{p_1p_2}=1$). Hence, it must be that $p_2$ is already
dead at that time, i.e.\
$$d^{(1)}_{p_2}<T_{p_1p_2}=d^{(2)}_{p_1}.$$
Consider now process 2. The fact that $d^{(2)}_{p_1}=T_{p_1p_2}$ implies
that
$$ T_{p_1p_2}=d^{(2)}_{p_1}< d^{(2)}_{p_2} .$$
Hence, in view of $d^{(1)}_{p_2}<d^{(2)}_{p_1},$
we have
$d^{(2)}_{p_2}>d^{(1)}_{p_2}$. Now, this reasoning can be
continued, leading to an infinite sequence of distinct points $p_n$
such that
$$
d^{(1)}_{p_1}>d^{(2)}_{p_1}>d^{(1)}_{p_2}>d^{(2)}_{p_3}>
d^{(1)}_{p_4}>\cdots
$$ and within this sequence ($i$ alternating between 1 and 2)
$d^{(i)}_{p_n}=T_{p_np_{n+1}}$ for $n=1,2,\ldots$. But this
sequence is exactly of the kind whose existence is a.s.\ denied by
Lemma \ref{finiteinvestigationlemma}.
\end{proof}

Below we will take $t_0=0$.
The Sheriff algorithm can be seen as a measurable mapping
from $(M(\Re^d\times\Re),(0,\infty)^\Nat,\set{0,1}^\Nat)$
to $M(\Re^d\times\Re)$ with
\begin{equation}
\label{eq:mapping}
\mathrm{Sheriff}(\Psi_{(0,t_1)},\set{T_{pq}},\set{I_{pq}}):=
\set{x_p,d_p}_{p\in\Psi_{(0,t_1)}},
\end{equation}
where, in the case $t_1<\infty$, we set $d_p=\infty$ for points living at
time $t_1$.

\subsection{More General Initial Conditions}
\label{generalinit}
The initial condition of the Sheriff algorithm was empty at time $t_0$
since the stacks were defined from the arrivals in $(t_0,t_1)$.
It will be useful below to extend this to an initial condition
made of a point process $\mathcal{Z}_0$ of nodes already present
(i.e.\ born) at time $t_0$ and having independent pairwise
random connections and killing direction variables
as those defined above. In this case, an initial stack
is built for each node
of $\mathcal{Z}_0\cup \Psi_{(t_0,t_1)}$, containing its sorted sentences.
If the point process $\mathcal{Z}_0$ satisfies the property in
Lemmas \ref{nonaccondprop} and \ref{finiteinvestigationlemma} 
namely if
\begin{enumerate}
\item for every $z$ in ${\cal Z}_0$, the set of all $T_{z,w}, w\in {\cal Z}_0$
which belong to $(t_0,t)$ is finite for $-\infty<t_0<t<\infty$;
\item there is no infinite sequence of points $z_1,z_2,\ldots$ of $\mathcal{Z}_0$
such that the sequence $T_{z_nz_{n+1}}$ be non-increasing,
\end{enumerate}
then there is no difficulty in running Sheriff on this
initial condition.\\

\noindent
{\bf {Throughout the paper, the initial condition will
be assumed to be a motion invariant
point process \cite{DalVJon:88} satisfying the 
conditions 1 and 2 given above.}} \\

Here are a few examples where this condition is satisfied.
If $\mathcal{Z}_0$ is Poisson, homogeneous and independent of $\Psi_{(t_0,t_1)}$,
this follows from Lemma \ref{finiteinvestigationlemma}.
By a direct monotonicity argument,
the same holds if $\mathcal{Z}_0$ is any compatible
thinning\footnote{By compatible thinning, we mean 
a thinning where the retention decisions are marks of the
point process.} of an independent homogeneous Poisson point process.
This compatible thinning 
can be based on the independent pairwise connections and killing directions.
In particular let $\Phi_{t_1}$ denote 
the point process, built by Sheriff, of nodes
living at time $t_1$ when the system
starts empty at time $t_0$. This is a motion
invariant thinning of $\Psi_{(t_0,t_1)}$ based on
these pairwise variables. 
One can hence apply Sheriff on $[t_1,t_2)$ to the initial
condition $\mathcal{Z}_0=\Phi_{t_1}$ for all $t_1<t_2<\infty$.

\subsection{The Double Card Version of Sheriff}
\label{doublecardrem}
The Sheriff algorithm could also be defined as follows:
in the initialization, for each
connection time $T_{pq}$, put a card $(p,q,T_{pq})$ in $p$'s
stack and a card $(q,p,T_{pq})$ in $q$'s stack (remember that
$T_{pq}=T_{qp}$). The values of the $I_{pq}$'s are not drawn beforehand,
so that we don't speak of death sentences but of {\em duel times}.
In Step 2 of Sheriff, copy (instead of move)
the card of the next point whose top card carries a time less than
$t_1$ and is not a death certificate
to the investigation stack. In step 3, there are three alternatives:
(i) if $q$'s stack has a death certificate on top, then $p$'s stack and the
investigation stack are popped; (ii) if $q$'s top card carries a duel time
less than $T_{pq}$, that card is copied to the investigation stack;
(iii) in the remaining case, $q$'s top card is $(q,p,T_{qp})$;
now the $I_{pq}$ variable is drawn, the loser's
top card is replaced by a death certificate,
and the killer's stack and the investigation stack are popped.
It is obvious that this variant, which will be referred to as
the {\em double card} version of the algorithm,
performs similarly to the initial one (although the
investigation order is not exactly the same). 

The double-card version makes it clear that the direction of the
interaction (i.e.\ the value of $I_{pq}$) need not be
specified before the step when it is really needed in the algorithm at
the realization of a duel. Another nice feature is that full information on
the $\set{T_{pq}}$ sequence remains in the stacks. Indeed, for an unrealized
duel, which can only happen when either or both duelists are dead before
the execution time, a copy of the card remains in the stack
of at least one of the two duelists.

\section{Initial Condition and Coupling}
\label{sec:init}

In this section, we investigate how additions to the initial condition
perturb the history of all other points. This is done
through a coupling, called Sheriff$^Z$, which allows one to jointly
build the histories with and without these additions. 
In Section \ref{sec:ergo}, we will leverage the finiteness of this
perturbation to construct the steady state through 
a coupling from the past.  

\subsection{Augmenting the Initial Condition}

Below, we consider two systems: (1)
that with an empty set of nodes
as initial condition (as in Sheriff); (2)
that with an initial condition consisting
of a point process $\mathcal{Z}_0$
in $\Re^d\times\set{0}$,
satisfying the conditions of Subsection \ref{generalinit} and
representing some additional set of points born at time 0.
The point process $\mathcal{Z}_0$ will be called the {\em augmentation}
point process.
The first case is a special case of the second one
(with $\mathcal{Z}_0=\emptyset$) and will be referred to
as the {\em non-augmented} case.

Our aim below will be to jointly build
two parallel executions of the killing history:
that with this augmentation and that without. 
The coupling consists of using the same sequences
of connections ($\{T_{pq}\}$) for common points (the points of $\Psi_{(0,t_1)}$).
The addition of $\mathcal{Z}_0$ has non-monotonic effects
on the life times of common points, with 
some points having their lifetime extended and others shortened.
In the algorithm described below,
at any given time, we call {\em zombies} the points that
are alive in the augmented process and are dead in the non
augmented process (i.e.\ with a death time
already determined in the non-augmented process and not yet
determined in the augmented one)\footnote{It makes sense to
call the points of ${\cal Z}_0$ zombies as well.}.
Conversely we will call
{\em antizombies} the points that are dead in the augmented
process and alive in the original process (i.e. with a death time 
already determined in the augmented process and not yet 
in the non-augmented one).

At any given time, zombies and antizombies will be called {\em special points}
and the other points will be called {\em regular}. 
The basic principles of the joint execution are as follows: 
\begin{itemize}
\item The killing of a regular point by another regular point
determines the death times of the former in the two processes (these
death times are equal).
\item If a zombie kills a regular point, this determines the death time of the latter
in the augmented process. This regular point becomes
an antizombie and is kept in the algorithm until its death time is determined in the
non-augmented process. 
\item If an antizombie kills a regular point, this determines the death time of the latter
in the non-augmented process. This point becomes
a zombie and is kept in the algorithm until its death time is determined in the 
augmented process. 
\item If a regular point kills a zombie, this determines the death time of the latter in the augmented process,
and this zombie can be forgotten as its two death times are now determined in both processes.
\item If a regular point kills an antizombie, this determines the death time of the latter in the non-augmented
process, and the antizombie can be forgotten for the same reasons as above.
\item If a zombie (resp. an antizombie)  kills another zombie (resp. antizombie),
this determines the death time of the latter in the augmented (resp. non-augmented)
process and the killed point can be forgotten.
\item Zombies and antizombies cannot kill each other as they belong
to different processes.
\end{itemize}

See Figure \ref{zombiefig} for an illustration.\\

\begin{figure}[t]
\begin{center}
\includegraphics[width=10.4cm]{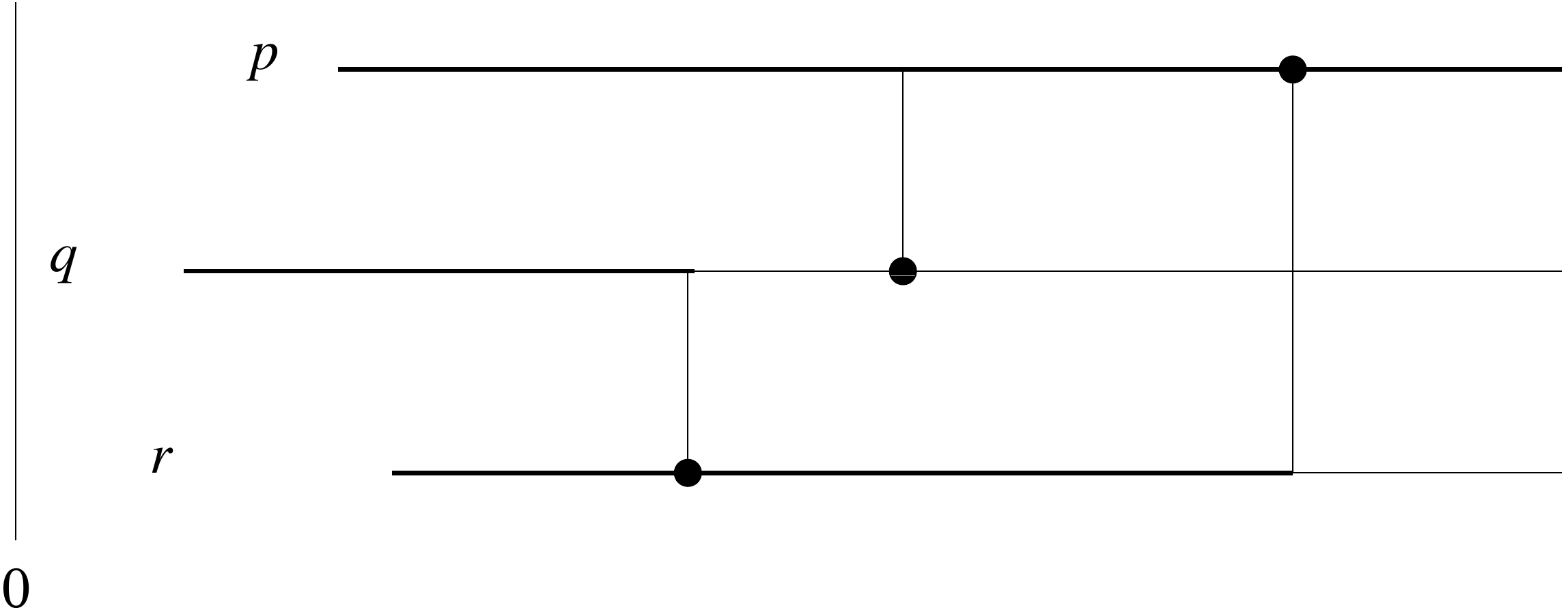}\\
\bigskip
\includegraphics[width=10.4cm]{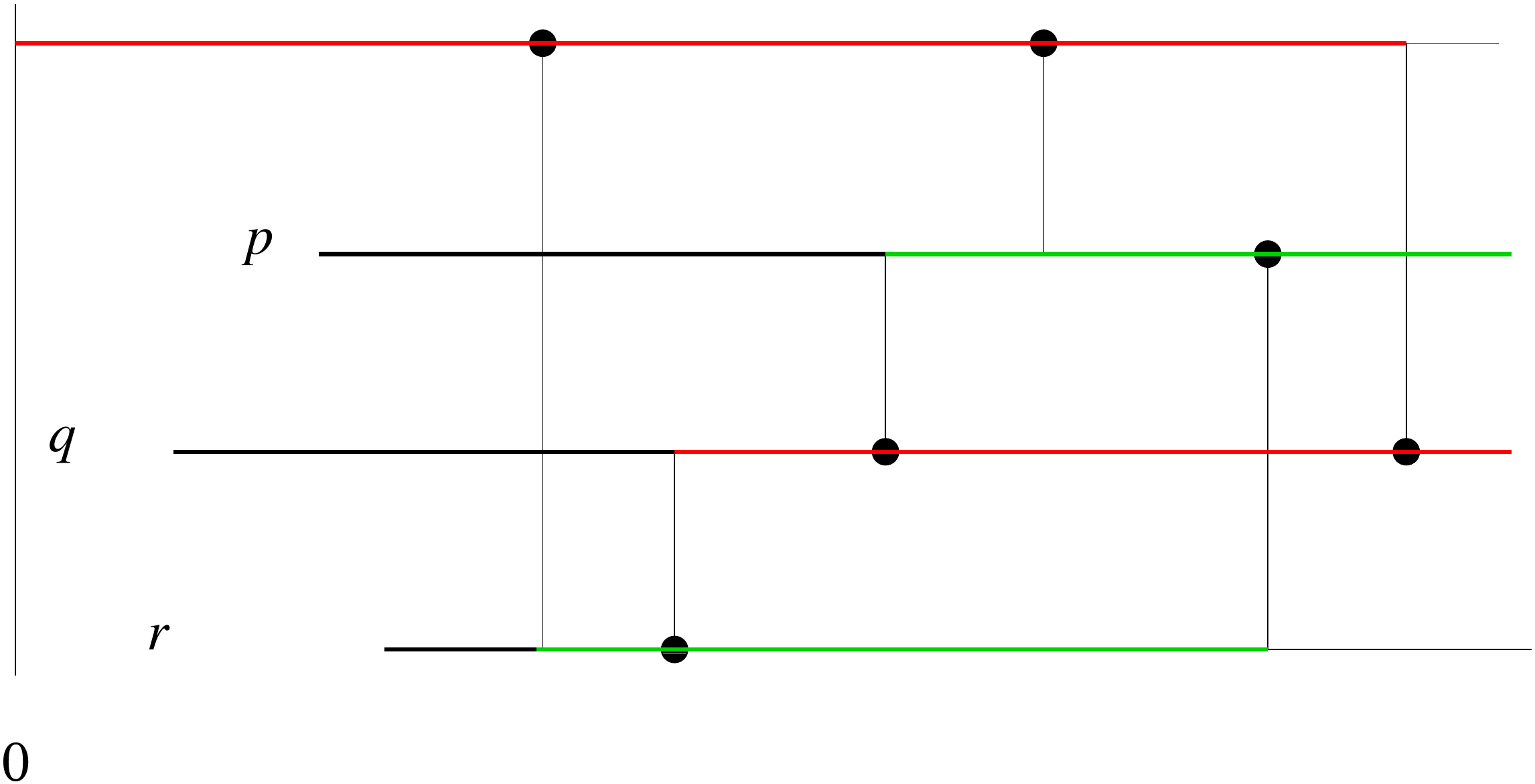}
\caption{Two coupled killing processes.
Time is on the $x$ axis; space on the $y$ axis.
Dots associated with vertical arrows indicate killers.
Upper figure: a killing process with three points.
Lower figure: same process with an added point born at time
zero. Colors: {\em black}: regular point; {\em red}: zombie;
{\em green}: antizombie.}
\label{zombiefig}
\end{center}
\end{figure}

The Sheriff$^Z$ algorithm described below generates
the announced coupling of the original and the augmented process.
It simultaneously builds two sequences
$\set{e_p}_{p\in\Psi_{(0,t_1)}}$ and
$\set{e'_p}_{p\in\mathcal{Z}_0\cup\Psi_{(0,t_1)}}$ 
that will be shown later to coincide with the
death sequences of the non-augmented and the augmented
systems, respectively. It does so
by maintaining the list of zombies and antizombies at all times.

\begin{algorithm}
\label{sheriff3}
\rm
\renewcommand{\labelenumi}{\arabic{enumi}.} 
\renewcommand{\theenumi}{\arabic{enumi}} 
Sheriff$^Z$ (``Sheriff, with zombies''). 
\newline
Input: $\mathcal{Z}_0$, $\Psi_{(0,t_1)}$,
$\set{T_{pq}}_{p,q\in\mathcal{Z}_0\cup\Psi_{(0,t_1)}}$,
$\set{I_{pq}}_{p,q\in\mathcal{Z}_0\cup\Psi_{(0,t_1)}}$.\\
Output: $\set{e_p\in [0,t_1]}_{p\in\cup\Psi_{(0,t_1)}}$,
$\set{e'_p\in [0,t_1]}_{p\in\mathcal{Z}_0\cup\Psi_{(0,t_1)}}$.
\begin{enumerate}
\item  Initialization:
\begin{itemize}
\item For each $p\in\mathcal{Z}_0\cup\Psi_{(0,t_1)}$, build a stack
$S_p$ of cards of the form $(p,q,T_{pq})$, where the $T_{pq}$ variables
are sorted in increasing order 
(earliest time on top)\footnote{We use here the double card version so that
for all cards of the form $(p,q,T_{pq})$ stored in $p$'s stack,
a card with the same data is also stored in $q$'s stack.};
\item The sheriff has an investigation stack (IS), initially empty.
\item For all $p$, $e_p:=e'_p:=\infty$;
\item For all $z\in\mathcal{Z}_0$, one maintains the point sets
  $\mathcal{A}(z)$ (``antizombies'') and $\mathcal{Z}(z)$ (``zombies'')
  offspring of $z$;
  initially, $\mathcal{A}(z):=\emptyset$ and
  $\mathcal{Z}(z):=\set{z}$; if at some time $p\in\mathcal{Z}(z)$ or
  $p\in\mathcal{A}(z)$, we define $\mathfrak{z}(p)=z$ (the value
  will be uniquely defined); denote (at all times)
  $\mathcal{A}=\cup_{z\in\mathcal{Z}_0}\mathcal{A}(z)$,
  $\mathcal{Z}=\cup_{z\in\mathcal{Z}_0}\mathcal{Z}(z)$;
\item Call a point $p$ {\em finished}, if either $p\in\mathcal{Z}_0$
  and $e'_p<\infty$, or $p\in\Psi_{(0,t_1)}$ and $e_p\vee
  e'_p<\infty$, or $S_p$'s top card has $T_{pq}\ge t_1$.
\end{itemize}
\item\label{sheriffcycle33} If IS is empty, the sheriff chooses,
  from a pre-defined
  ordering of all points, the first unfinished point $p$ such that
  the top card of $S_p$ has $T_{pq}<t_1$, and copies this card to IS.
\item\label{sheriffcycle34} The sheriff looks at the top card of IS,
say $(p,q,T_{pq})$.
\begin{itemize}
\item if $q$ is finished, he pops both $S_p$ and IS; goes to Step
  \ref{sheriffcycle33};
\item if $S_q$'s top card has $T_{qr}<T_{pq}$, he copies this
last card to IS; goes
  to Step \ref{sheriffcycle34}.
\end{itemize}
\item The sheriff does one of the following
(if the appropriate case is missing,
  interchange $p$ and $q$): \begin{description}
\item[\rm $p,q\in \mathcal{A}^c\cap \mathcal{Z}^c$:]\hspace*{\fill}
  \begin{itemize}
  \item if $I_{pq}=1$, $e_p:=e'_p:=T_{pq}$; pops $S_q$;
  \item if $I_{pq}=0$, $e_q:=e'_q:=T_{pq}$; pops $S_p$;
  \end{itemize}
\item[\rm $p\in \mathcal{Z}$ and $q\in \mathcal{A}^c\cap \mathcal{Z}^c$:]\hspace*{\fill}
  \begin{itemize}
  \item if $I_{pq}=1$, $e'_p:=T_{pq}$;
    $\mathcal{Z}(\mathfrak{z}(p))
:=\mathcal{Z}(\mathfrak{z}(p))\setminus\set{p}$;
    pops $S_q$;
  \item if $I_{pq}=0$, $e'_q:=T_{pq}$;
    $\mathcal{A}(\mathfrak{z}(p))
:=\mathcal{A}(\mathfrak{z}(p))\cup\set{q}$; pops $S_p$ and $S_q$;
  \end{itemize}
\item[\rm $p\in \mathcal{A}$ and $q\in \mathcal{A}^c\cap \mathcal{Z}^c$:]\hspace*{\fill}
  \begin{itemize}
  \item if $I_{pq}=1$, $e_p:=T_{pq}$;
    $\mathcal{A}(\mathfrak{z}(p))
:=\mathcal{A}(\mathfrak{z}(p))\setminus\set{p}$;
    pops $S_q$;
  \item if $I_{pq}=0$, $e_q:=T_{pq}$;
    $\mathcal{Z}(\mathfrak{z}(p))
:=\mathcal{Z}(\mathfrak{z}(p))\cup\set{q}$; pops $S_p$ and $S_q$;
  \end{itemize}
\item[\rm $p,q\in \mathcal{Z}$:]\hspace*{\fill}
  \begin{itemize}
  \item if $I_{pq}=1$, $e'_p:=T_{pq}$; $\mathcal{Z}(\mathfrak{z}(p))
:=\mathcal{Z}(\mathfrak{z}(p))\setminus\set{p}$; pops $S_q$;
  \item if $I_{pq}=0$, $e'_q:=T_{pq}$; $\mathcal{Z}(\mathfrak{z}(q))
    :=\mathcal{Z}(\mathfrak{z}(q))\setminus\set{q}$; pops $S_p$;
  \end{itemize}
\item[\rm $p,q\in \mathcal{A}$:]\hspace*{\fill}
  \begin{itemize}
  \item if $I_{pq}=1$, $e_p:=T_{pq}$; $\mathcal{A}(\mathfrak{z}(p))
:=\mathcal{A}(\mathfrak{z}(p))\setminus\set{p}$; pops $S_q$;
  \item if $I_{pq}=0$, $e_q:=T_{pq}$; $\mathcal{A}(\mathfrak{z}(q))
:=\mathcal{A}(\mathfrak{z}(q))\setminus\set{q}$; pops $S_p$;
  \end{itemize}
\item[\rm $p\in \mathcal{Z}$ and $q\in \mathcal{A}$:]\hspace*{\fill}
  \begin{itemize}
  \item pops $S_p$ and $S_q$.
  \end{itemize}
\end{description}
\item The sheriff pops IS and goes to Step \ref{sheriffcycle33}.
\end{enumerate}
\end{algorithm}

\begin{rem}
If the set $\mathcal{Z}_0$ is empty, Sheriff$^Z$ reduces to the double-card version of Sheriff (see Subsection \ref{doublecardrem}); in this case,
the first case in Step 4 is always met.
\end{rem}

\begin{rem}
In the second bullet of cases 2 and 3 in Step 4, one discards the top
cards of both $p$ and $q$ because $q$ does not kill $p$ but only
labels it, and the connection between the two can be forgotten.
\end{rem}

\subsection{Properties of the Sheriff$^Z$ Map}
Let us now see in detail what Sheriff$^Z$ does. Let
$$
\set{x_p,d_p}_{p\in\Psi_{(0,t_1)}}
=\mathrm{Sheriff}(\Psi_{(0,t_1)},\set{T_{pq}},\set{I_{pq}}),
$$ where $\set{T_{pq}}$ and $\set{I_{pq}}$ are indexed by
$p,q\in\Psi_{(0,t_1)}$
and let
$$
\set{x,p,d'_p}_{p\in\mathcal{Z}_0\cup\Psi_{(0,t_1)}}
=\mathrm{Sheriff}(\mathcal{Z}_0\cup\Psi_{(0,t_1)},\set{T_{pq}},\set{I_{pq}}),
$$ where $\set{T_{pq}}$ and $\set{I_{pq}}$ are now indexed by
$p,q\in\mathcal{Z}_0\cup\Psi_{(0,t_1)}$.
In these last definitions, we use the same sequences 
$\set{T_{pq}}$ and $\set{I_{pq}}$ as in Sheriff$^Z$. Then we have:

\begin{thm}\label{zombiesheriffthm}
Under Assumptions 0--1,
the following claims hold almost surely:\hspace*{\fill}
\begin{enumerate}
\item\label{zshclaimeq}The algorithm Sheriff$^Z$ runs unambiguously 
and every point gets finished (in the sense defined at the end
of the initialization) in finite time. For all $p\in\Psi_{(0,t_1)}$, 
$e_p=d_p$, whereas for all $p\in\mathcal{Z}_0\cup\Psi_{(0,t_1)}$, $d'_p=e'_p)$.
\item\label{zshclaimaz}For each $z\in\mathcal{Z}_0$, the Sheriff$^Z$
  algorithm (implicitly) generates the set-valued stochastic processes
  $(A_t(z))_{t\in[0,t_1)}$ and
  $(Z_t(z))_{t\in[0,t_1)}$ representing, respectively, the
  antizombies and zombies originating from $z$ and living at time
  $t$. These sets satisfy the conditions
\begin{subequations}
\label{azconds}
\begin{gather}
  p\in A_t(z)\mbox{ for some }z\quad\Leftrightarrow \quad d'_p\le t< d_p\text{,}\\
  p\in Z_t(z)\mbox{ for some }z\quad\Leftrightarrow \quad d_p\le t< d'_p\text{.}
\end{gather}
\end{subequations}
The ``families'' of offsprings
$$
O_{t_1}(z)=\bigcup_{t\in[0,t_1)}A_t(z)
\cup\bigcup_{t\in[0,t_1)}Z_t(z)
$$
of distinct $z$'s are disjoint.
\end{enumerate}
\end{thm}

\begin{proof}
First note that if, in any phase of the algorithm, a point $p$ belongs
to $A(z)\cup Z(z)$ for some $z\in\mathcal{Z}_0$,
then $z$ is unique and we can thus denote it as
$\mathfrak{z}(p)$. Indeed, this holds true for the initial situation where
$A(z)\cup Z(z)=\set{z}$, and any given point can be
added to some $A(z)\cup Z(z)$ only once. Thus, the
steps of the algorithm are unambiguously defined. 

Second, note that we now use double cards as discussed in Subsection
\ref{doublecardrem}. All points get finished (a.s.) by the
argumentation used for proving the same for Sheriff.

Since the algorithm clearly fixes the times when a point becomes or
ceases to be a zombie or antizombie, it is obvious that the processes
$(A_t(z))_{t\in[0,t_1)}$ and
$(Z_t(z))_{t\in[0,t_1)}$ are well defined.

A further examination of
the algorithm yields the conditions (\ref{azconds}). We also see why
zombies and antizombies don't interact (last case of step 4): if
$p\in A_t(z)$ and $q\in Z_t(z')$, at time $t$, $p$ is
dead in the augmented scenario and $q$ is dead in the original
scenario. Claim \ref{zshclaimaz} is now proven, since the last
subclaim just states the uniqueness discussed already in the beginning of the proof.

To show that $d_p=e_p$ for all $p$, note that zombies are
points that already have the $e$-value set. When a zombie kills a
regular point, the latter receives an $e'$-value and becomes an antizombie,
but this has no effect on the setting of subsequent $e_p$-values. If
zombies were considered as finished (that is, dead), the $e_p$-values
would be the same.

This is illustrated in Figure \ref{zombiebisfig}.

\begin{figure}[t]
\begin{center}
\hspace*{-4mm}\includegraphics[width=9.4cm,angle=-90]{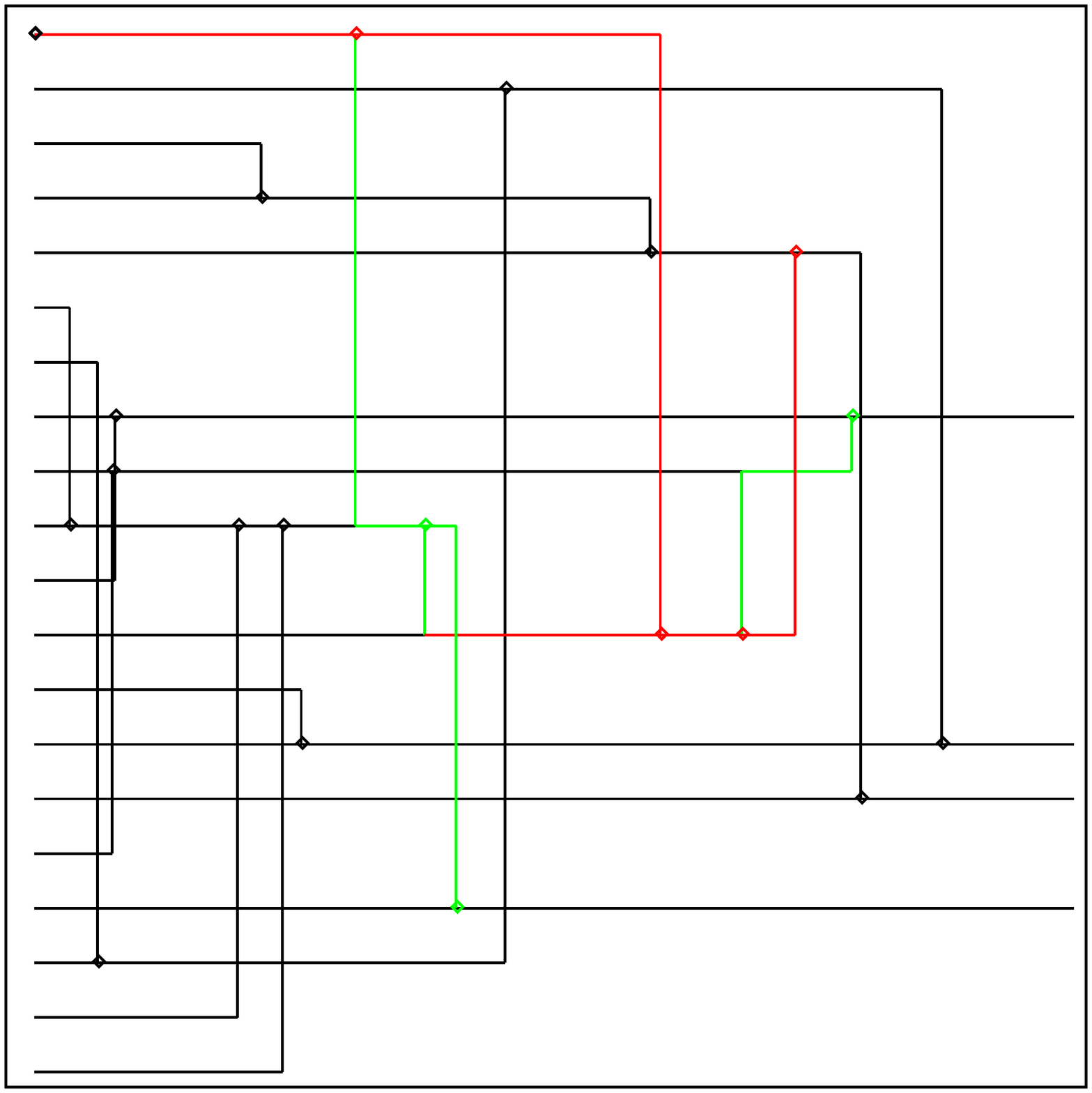}
\caption{
Time is on the $x$ axis; space on the $y$ axis.
All nodes are born at the same time  
to increase readability. Dots indicate killers.
Compared to Figure \ref{zombiefig}, only relevant vertical edges are
represented.
Colors: {\em black}: regular point; {\em red}: zombie;
{\em green}: antizombie. Black plus green define the $e$ variables.
Black plus red define the $e'$ variables.}
\label{zombiebisfig}
\end{center}
\end{figure}

Similarly, the antizombies have no effect on the setting of
$e_p$-values --- if we are interested only in the latter, antizombies
could as well be considered as finished. This completes the proof of
Claim \ref{zshclaimeq}.
\end{proof}

\begin{thm} \label{zshclaimherfin}
Under Assumptions 0--1,
the families $O_{t_1}(z)$ are finite; this also holds
in the case $t_1=\infty$.
\end{thm}

\begin{proof}
Clearly it suffices to consider the case $t_1=\infty$. Define the
filtration $\mathbb{F}=(\mathcal{F}_t)_{t\ge0}$ as
$$
\mathcal{F}_t=\sigma(\mathcal{Z}_0\cup\Psi_{(0,t)})\vee
\sigma((T_{xy},I_{xy}):\ x,y\in\mathcal{Z}_0\cup\Psi_{(0,t]},\ T_{xy}\indic{T_{xy}\le t}).
$$ Obviously, if
$x,y\in\mathcal{Z}_0\cup\Psi_{(0,\infty)}$, then $T_{xy}$ is an
$\mathbb{F}$-stopping time. Further, $I_{xy}$ is not
$\mathcal{F}_{T_{xy}-}$-measurable, but
$\mathcal{F}_{T_{xy}}=\mathcal{F}_{T_{xy}-}\vee\sigma(I_{xy})$.

Let us fix a point $z\in\mathcal{Z}_0$ and let
$S_t(z)=A_t(z)\cup Z_t(z)$. For $t\ge0$, define
the `set of relevant points' as
$$ U_t(z)=\eset{x\in\mathcal{Z}_0\cup\Psi_{(0,\infty)}}{\exists
  y\in S_t(z): T_{xy}>t}.
$$ We define inductively a non-decreasing sequence of stopping times
$T_n$ and a sequence of random variables $J_n\in\set{-1,0,1}$ by setting
$T_0=J_0=0$ and
\begin{eqnarray*}
T_{n+1}=\left\{
\begin{array}{ll}
\infty,&\ \mbox{if }T_n=\infty,\\
\inf\eset{T_{xy}}{x\in U_{T_n},\ y\in S_{T_n}(z)},&\ \mbox{otherwise,}
\end{array}
\right.
\end{eqnarray*}
with $\inf\emptyset=\infty$. Assume that, for some $n$,
$S_{T_n}$ is finite and non-empty. By Lemma \ref{nonaccondprop}
we have, a.s., $T_n<T_{n+1}<\infty$, and $S_t$ does not change
on $[T_n,T_{n+1})$. Let now $T_{n+1}=T_{xy}$, where
  $y\in S_{T_n}(z)$, and if $x\in S_{T_n}(z)$,
  we choose for definiteness $x$ as being farther from the origin than
  $y$. At time $T_{n+1}$, one of the following takes place:
\begin{itemize}
\item{Case 0:} $x$ is finished before time $T_{xy}$, or $x$ is, before
  $T_{n+1}$, a special point of the kind opposite to that of $y$; then
  $S_{T_{n+1}}(z)=S_{T_n}(z)$.
\item{Case 1:} $x$ is a regular point. Then the chances, determined by
  $I_{xy}$, are $\frac12$ that
  $S_{T_{n+1}}(z)=S_{T_n}(z)\cup\set{x}$ (with $x$
  being killed by a zombie or an antizombie and being transformed to
  the opposite kind) and $\frac12$ that
  $S_{T_{n+1}}(z)=S_{T_n}(z)\setminus\set{y}$ ($y$
  being killed by $x$).
\item{Case 2:} $x$ is a special point of the same kind as $y$, and
  $x\not\in S_{T_n}(z)$. Now the chances are
  $\frac12,\frac12$ that
  $S_{T_{n+1}}(z)=S_{T_n}(z)$ or
  $S_{T_{n+1}}(z)=S_{T_n}(z)\setminus\set{y}$.
\item{Case 3:} $x$ and $y$ are special points of same kind, and
  $x\in S_{T_n}(z)$. Then
  $S_{T_{n+1}}(z)=S_{T_n}(z)\setminus\set{x}$ or
  $S_{T_{n+1}}(z)=S_{T_n}(z)\setminus\set{y}$,
  depending on $I_{xy}$.
\end{itemize}
In Case 0 we set $J_{n+1}=0$, and in the remaining cases
$J_{n+1}=-1+2I_{xy}$. On the other hand, if $T_{n+1}=\infty$, we set
$J_{n+1}=0$, and interpret
$S_{T_{n+1}}(z)=S_{T_n}(z)$. Note that
$S_{T_{n+1}}$ is finite in every case, and
\begin{equation}
s_n:=|S_{T_n}(z)|\le1+\sum_{k=1}^n J_k,\quad n\ge0.
\label{rwbound}
\end{equation}
In fact, the sequence $s_n$ is an
$(\mathcal{F}_{T_n})$-supermartingale. If $T_n=\infty$ for some $n$,
we clearly have $|O_\infty(z)|<\infty$. If $T_n<\infty$ for
all $n$ but $J_n\not=0$ for only finitely many $n$, then
$O_{T_n}$ remains unchanged for $n\ge n_0$, and
$T_n\to\infty$ by Lemma \ref{nonaccondprop}, thus again
$|O_\infty(z)|<\infty$. Finally, assume that we have with
positive probability $T_n<\infty$ for all $n$ and $J_n\not=0$ for
infinitely many $n$. Now, such $J_n$s are independent random variables
taking values $\pm1$ with probabilities $\frac12,\frac12$. Since a
symmetric random walk on $\mathbb{N}$ hits zero with probability 1,
there is a finite random number $c$ such that $1+\sum_{k=1}^c
J_k=0$. Now, by (\ref{rwbound}), $s_c=0$, and we get
$T_{c+1}=\infty$, which contradicts the assumption.
This concludes the proof.
\end{proof}

\subsection{Section Summary}
\label{sum1}
Let us summarize this section by focusing on the case where
the augmentation point process ${\cal Z}_0$ is translation
invariant in $\R^d$. We established the following results: 
\begin{enumerate}
\item Theorem
\ref{zombiesheriffthm} uniquely defines
the marked point process $\widetilde \Phi_t$ 
of nodes which are not finished at time $t<\infty$ under Sheriff$^Z$;  
the marks belong to the set $\{{\cal R}, {\cal A}, {\cal Z}\}$.
The points with mark ${\cal R}$ are regular points, which
are alive both in the augmented and the non-augmented
processes, whereas
those with mark ${\cal A}$ (resp. ${\cal Z}$) are antizombies (resp. zombies)
with a life shorter (resp. longer) in the augmented process
compared to the non-augmented one.
\item Theorem \ref{zombiesheriffthm}
also shows that the  points of $\widetilde \Phi_t$
with marks in ${\cal R}\cup {\cal Z}$
form a stationary point process $\Phi'_t$ which coincides with
that built by Sheriff at time $t$ when the initial condition is ${\cal Z}_0$ .
Similarly, the points of $\widetilde \Phi_t$ with marks
in ${\cal R}\cup {\cal A}$ form a stationary point process $\Phi_t$
which coincides with
that built by Sheriff at time $t$ when the initial condition is $\emptyset$.
\item
Theorem \ref{zshclaimherfin} shows that the set of special points
of $(\widetilde \Phi_t)_{t\ge0}$ which are
offsprings of a given point $z\in {\cal Z}_0$ has a finite cardinality a.s.  
This collection of sets is translation invariant.
\end{enumerate}

\section{Non Degeneracy of Transient Densities}
\label{sec:globtight}

From now on, the augmentation point process ${\cal Z}_0$
is assumed to be motion invariant in $\R^d$ and to satisfy
the assumptions of Subsection \ref{generalinit}.
\subsection{Tightness} \label{sec:tight}
This section contains a simple stochastic comparison
argument showing that the stochastic processes
built by Sheriff are tight, which in turn implies
that densities admit a uniform upper bound.

Let us define a {\em mutual-service process with parameters $(\tilde\lambda,\tilde\mu)$}
as the birth-death process whose birth and death intensities in state $j$ are
$$
\lambda_j\equiv\tilde\lambda,\quad\mu_j=j(j-1)\tilde\mu,\quad j=0,1,2\ldots.
$$
Note that although a mutual-service process may start from state 0, it cannot reach 0 from any
other state $j>0$. 

\begin{prop}
\label{lem:boundspec}
Under Assumptions 0--2,
if ${\cal Z}_0$ is a translation invariant thinning of
a Poisson process on $\Re^d$, then the process $(\Phi_t)_{t\ge0}$
built by Sheriff satisfies the following properties:
\begin{enumerate} 
\item\label{statclaim} For any $t>0$, $\Phi_t$ is spatially
stationary and ergodic.
\item\label{tightclaim} Let $b>0$ be sufficiently small
to satisfy $f(b\sqrt{d})>0$;  
tessellate $\R^d$ into cubes $C_i$ of side $b$ indexed by
$i\in \Z^d$, assuming that the
center of $C_0$ is the origin. Then there exists a mutual service process 
$U(i)$ with parameters 
$\tilde\lambda=\lambda b^d,$
$\tilde\mu=2f(b\sqrt{d})$,
such that, a.s.,
\begin{equation}
\label{phiboundeq}
\Phi_t(C_i)\le U_t(i),\quad t\ge0,\ i\in\Z^d,
\end{equation}
and the processes $U(i)$ are independent given their initial states
\begin{equation}
\label{inicoupleq}
U_0(i)={\cal Z}_0(C_i),\quad i\in\Z^d.
\end{equation}
\item\label{betabound} The intensities $\beta_{\Phi_t}$ satisfy the bound
\begin{equation}
\label{eq:intbound}
\beta_{\Phi_t}<c, \quad  t\ge0,
\end{equation}
with $c$ a finite constant.
\item\label{phimombound}  For all positive integers $k$
and for all bounded Borel sets $C$,
$\E (\Phi_t(C)^k)$ is uniformly bounded in $t$.
\end{enumerate}
\end{prop}
\begin{proof}
Claim \ref{statclaim}: the (space) stationarity and the ergodicity
follow from the fact that the point process $\Phi_t$ is 
a translation invariant thinning of an independently
marked stationary and ergodic point process.

Claim \ref{tightclaim}: by a classical coupling argument,
we can construct the processes $U_t(i)$ on an extension
of the probability space of $(\Phi_t)$ so that (\ref{phiboundeq})
and the conditional independence hold. Set the initial states of 
$U_t(i)$ according to (\ref{inicoupleq}). For each $i\in\Z^d$,
we can obviously make the up-jumps of $\Phi(C_i)$ and $U_t(i)$ identical.
For down-jumps, assume that 
$\Phi_t\cap C_i=\set{X_1,\ldots,X_m}$ with $m\ge 2$,
and that $U_t(i)=j\ge m$. Given $\Phi_t$,
the times $T_{X_{i_1}X_{i_2}}$ are independent exponentially
distributed random variables with parameters
$2f(\|X_{i_1}-X_{i_2}\|)\ge\tilde\mu$, respectively. If $j=m$,
$$
\min_{i_1} d_{X_{i_1}}\le\min_{{i_1},{i_2}}T_{X_{i_1}X_{i_2}}\buildrel(\mathrm{st})\over\le\mbox{Exp}(j(j-1)\tilde\mu),
$$
with {\em st} denoting stochastic ordering.
The claim then follows from these observations.

Claims \ref{betabound} and \ref{phimombound} follow from Claim
\ref{tightclaim} since,
except for state 0, a mutual-service process is
dominated by an M/M/$\infty$ queue with the same parameters,
whose stationary distribution is Poisson.
\end{proof}

\begin{rem}
For all thinned Poisson initial conditions ${\cal Z}_0$  (such initial
conditions satisfy the above assumptions), one can adapt the last
proof and obtain analogues of Proposition \ref{lem:boundspec} for
\begin{enumerate}
\item The point process $\Phi'_t=\Phi_{t,\mathcal{Z}_0}$ built by Sheriff;
\item The point process $\widetilde \Phi_t$ built by Sheriff$^Z$;
for showing this last property, one can use the fact
that $\widetilde \Phi_t$ is bounded from above
by the superposition of the point processes $\Phi_t$
and $\Phi'_t$, which both satisfy the desired properties.
\end{enumerate}
\end{rem}
\begin{lem}
\label{lem:boundspecpres} Let $ \E^0_\chi $ stand for the Palm probability of a point process $ \chi$.
Under Assumptions 0--3, there exists a finite constant $c$ such that
\begin{equation}
\label{eq:intboundpress}
\beta_{\Phi_t}\E^0_{\Phi_t} \sum_{X\in \Phi_t} f(\|X\|) \le c, \quad \forall 
t\in \R.
\end{equation}
More generally, for all positive integers $k$, there exists a finite
constant $c_k$ such that
\begin{equation}
\label{eq:intboundpressmom}
\beta_{\Phi_t}\E^0_{\Phi_t} \left[\left(\sum_{X\in \Phi_t} f(\|X\|)\right)^k\right]
\le c_k, \quad \forall 
t\in \R.
\end{equation}
\end{lem}

\begin{proof}
Using the product form upper bound described above, we get
\begin{eqnarray*}
\beta_{\Phi_t} \E^0_{\Phi_t}
\sum_{X\in \Phi_t} f(\|X\|)
&  =& \frac 1 {b^d} \E 
\sum_{X\in \Phi_t\cap C_0} 
\sum_{Y\ne X\in \Phi_t} f(\|X-Y\|)\\
& \le & \frac K {b^d}  \E  \Phi_t(C_0)^2 \\ & & +\frac 1 {b^d}
\E \sum_{X\in \Phi_t\cap C_0}  \sum_{i\ne 0}
\sum_{Y \in \Phi_t\cap C_i}  f(\|X-Y\|)\\
&  \le & \frac K {b^d}  \E  \Phi_t(C_0)^2 + \frac {(\E U_t(0))^2} {b^d} 
\sum_{i\ne 0} f(d_i+),
\end{eqnarray*}
where $K$ is the upper-bound on $f$ (Assumption 3),
$U(0)$ is the mutual service process defined in the proof of Proposition \ref{lem:boundspec}, $d_i$ is the distance from $C_i$ to $C_0$, and $d_i+$ stands for the right-hand limit (to handle the case where $ d_i=0 $).
There exists a constant $H>1$ and a ball $B$ centered
in the origin such that for all $i$ with
$C_i$ not included in $B$ and for all $x\in C_i$,
$\|x\| \le H d_i$. If $ \nu_d $ denotes the volume of a unit ball, this in turn implies that
$$ 
\frac 1 {b^d} \sum_{i\ne 0, C_i\notin B} f(d_i) \le
\int_{\R^d} f\left(\frac{\|x\|} H\right) dx =
d \nu_{d} \int_{r>0} f\left(\frac r H\right) r^{d-1} dr =
 H^d a < \infty.
$$ 
The proof of the first statement is then concluded from the
second statement of Proposition \ref{lem:boundspec}
and from the fact that the moments of $U_t(0)$ are uniformly bounded.

For $k\ge 1$, using again the product form upper bound, we get
\begin{eqnarray*}
\beta_{\Phi_t} \E^0_{\Phi_t}
\left(\sum_{X\in \Phi_t} f(\|X\|)\right)^k
&  =& \frac 1 {b^d} \E 
\sum_{X\in \Phi_t\cap C_0} 
\left(\sum_{Y\ne X\in \Phi_t} f(\|X-Y\|)\right)^k\\
&  \le &
\frac {K^k}{b^d} \E  (\Phi_t ( C_0))^{k} \\
& &\hspace{-2cm} + \frac 1 {b^d}
\E  U_t (0) 
\sum_{n_i\ge 0: \sum_{i\ne 0} n_i=k}\ \ \prod_{i\ne 0: n_i>0}
f^{n_i}(d_i) \E \Phi_t ( C_i)^{n_i}.
\end{eqnarray*}
Using now the fact that there exists a constant $J\ge 1$ such that
$$\E \Phi_t ( C_i)^{n} \le J (\E \Phi_t ( C_i))^n$$
uniformly in $i$, $t$ and $n\le k$ (Proposition \ref{lem:boundspec}),
we get that
\begin{eqnarray*}
\beta_{\Phi_t} \E^0_{\Phi_t}
\left(\sum_{X\in \Phi_t} f(\|X\|)\right)^k
&  \le & \frac {K^k} {b^d} 
\E  (\Phi_t ( C_0))^{k} \\ & & + \frac J {b^d}
(\E  U_t (0))^{k+1}  \left(\sum_{i\ne 0} f(d_i+)\right)^k
\end{eqnarray*}
and the proof of the second assertion then follows from the
finiteness of $a$ as above.
\end{proof}
The results of the last lemma extend to the point process
$\Phi'_t=\Phi_{t,\mathcal{Z}_0}$ built by Sheriff for all
thinned Poisson initial conditions.

\subsection{Positiveness} \label{sec:pos}
We now prove that for all finite $t$, the densities of all our point
processes are positive.  We denote by ${\cal R}_t$ (resp. ${\cal
  Z}_t$, ${\cal A}_t$ and ${\cal S}_t$) the stationary point process
of regular points (resp. zombies, antizombies and special points) built by
Sheriff$^Z$ at time $t$.  For each of these point processes, say
${\cal X}_t$, we denote its intensity by $\beta_{{{\cal X}}_t}$.
\begin{lem}
\label{lem:pos}
Make Assumptions 0--2, and let $\beta_{{{\cal Z}}_0}>0$.  Then
$\beta_{{{\cal X}}_t}>0$ for all finite $t$, with ${\cal X}_t={\cal
  R}_t, {\cal Z}_t, {\cal A}_t, {\cal S}_t$.
\end{lem}
\begin{proof}
Let $t$ be fixed and finite.
We start with ${\cal X}_t={\cal Z}_t$.
Let $z$ be a typical point of ${\cal Z}_0$.
The total number of points of $\Psi(0,t)$ that have a connection
to $z$ that takes place before time $t$ is (stochastically) bounded from
above by a Poisson random variable with parameter
$$ \int_{\R^d} (1-e^{-f(\|x-z\|)}) \lambda t dx
\le \lambda t a.$$
The total number of points of ${\cal Z}_0$ that have a connection
to $z$ is also finite by assumption (Item 2 in Subsection \ref{generalinit}). 
Hence the probability that all duels involving $z$ and taking
place before time $t$ are oriented in such a way that 
$z$ survives is positive. This shows that $\beta_{{\cal Z}_t}>0$
and also that $\beta_{{S}_t}>0$.

In order to prove the result for ${\cal X}_t={\cal A}_t$,
we pick $\epsilon <t$ and
we use arguments similar to those above to show that the probability that
(1) $z$ survives until time $t-\epsilon$;
(2) $\Psi(t-\epsilon,t)$ 
brings one arrival which kills $z$ (which becomes
an antizombie); and (3) the latter survives until time $t$,
is positive.

In order to prove the result for ${\cal X}_t={\cal R}_t$,
we pick $\epsilon <t$ and we look at the arrivals
of $\Psi(0,\epsilon)$ in each box $C_i=i+[0,1)^d$, where 
$k$ ranges over $\Z^d$. For those boxes that have at least
one arrival, pick the first of them. This defines a point
process. For a typical point of this point process, say $r$,
we use arguments similar to those above to show that the probability that
$r$ survives until time $t$
is positive.
This shows that $\beta_{{\cal R}_t}>0$.
\end{proof}
\begin{rem}
\label{rembphi}
It follows from the last lemma and from 2. in Section \ref{sum1}
that the densities
$\beta_{\Phi_t}$ and $\beta_{\Phi'_t}$ are also positive for all finite $t$.
\end{rem}

\section{Differential Equations for Transient Moment Measures}
\label{sec:compl}

The setting of this section is the same as
that of Section \ref{sec:init}, with the empty
and augmented initial conditions. We complement the 
result on the finiteness of the special
points stemming from a single point (Theorem 
\ref{zshclaimherfin}) by a set of 
differential equations on the densities and higher order 
moment measures of nodes of all types.
These equations will be needed in the coupling from the past arguments
of the next section.

We assume that the augmentation ${\cal Z}_0$ is a motion invariant
point process satisfying the assumptions of Subsection \ref{generalinit}.
The default setting is that
Assumptions 0--3 hold.
\subsection{Densities}
\subsubsection{Sheriff}
Let $\beta_{{\Phi'}_t}$ denote the density of the point
process $\Phi'_t$ built
by Sheriff for the initial condition ${\cal Z}_0$.
Let  $\E^0_{\Phi'_t}$ denote its Palm probability
(since $\beta_{{\Phi'}_t}>0$, see Remark \ref{rembphi}, the
latter is well defined). For all
$x\in \R^d$, let
\begin{equation}
\pi_{\Phi'_t}(x)=\sum_{X\in \Phi'_t} f(\|X-x\|).
\label{eq:defpres-sh}
\end{equation}
This quantity can intuitively be interpreted as the
\emph{death pressure} exerted by $\Phi'_t$ at $ x $.
The death rate of a typical node living at time $t$ is
$\E^0_{\Phi'_t} \pi_{\Phi'_t} (0)$.
The following equation is proved in Appendix \ref{appdiff}.
\begin{eqnarray}
\label{eq:difpsi'ante}
\frac d {dt} \beta_{{\Phi'}_t} & = & 
\lambda 
-\beta_{{\Phi}'_t} \E^0_{{\Phi}'_t} \pi_{{\Phi}'_t} (0).
\end{eqnarray}
From Proposition \ref{lem:boundspec} and Property 1
at the end of Section \ref{sec:tight}, the term
$\beta_{\Phi'_t}  \E^0_{{\Phi}'_t} \pi_{{\Phi}'_t} (0)$
that we find on the R.H.S. of this differential equation 
is uniformly bounded in $t$.

From the fact that $\beta_{\Phi'_t}$ is uniformly
bounded and from (\ref{eq:difpsi'ante}), we also get that
\begin{eqnarray}
\frac 1 t \int_0^t \beta_{\Phi'_u} \E^0_{\Phi'_u} \pi_{\Phi'_u} (0) du & = & \lambda +o(1)
\end{eqnarray}
as $t$ tends to infinity.

\subsubsection{Sheriff$^Z$}
For all $t>0$, for each of the point processes
${\cal X}_t={\cal R}_t$, ${\cal Z}_t$, ${\cal A}_t$ or ${\cal S}_t$, 
since $\beta_{{X}_t}>0$ (Lemma \ref{lem:pos}), the
Palm probability $\E^0_{{\cal X}_t}$ w.r.t. ${\cal X}_t$ is well defined.

Assuming ${\cal X}_t$ is the point process of nodes that interact
with a node located at $x$, we define the
death pressure exerted by the nodes of ${\cal X}_t$
on $x$ as
\begin{equation}\pi_{{\cal X}_t}(x):=\sum_{X\in {\cal X}_t} f(\|X-x\|),
\quad \pi_{{\cal X}_t}:=\pi_{{\cal X}_t}(0).
\label{eq:defpres}
\end{equation}
Since zombies and antizombies do not interact, we refine this definition in the case of $\mathcal{S}_t$ as follows:
\begin{eqnarray}
\label{specialpressuredef}
 \pi_{{\cal S}_t}(x)=\left\{
\begin{array}{l}
\displaystyle\sum_{y\in {\cal Z}_t }
 f(|x-y|)\text{ if } x\in\mathcal{Z}_t\text{,}\\
\displaystyle\sum_{y\in {\cal A}_t }
 f(|x-y|)\text{ if } x\in\mathcal{A}_t\text{,}\\
\displaystyle\sum_{y\in {\cal S}_t }
 f(|x-y|)\text{ otherwise.}
\end{array}
\right.
\end{eqnarray}
This must be taken into account when working with the Palm probability of special points,
since the point at origin may be of either type. Consequently, the general relation
\begin{equation}
\label{eq:zas}
\E^0_{{\cal S}_t}= 
\frac{\beta_{{\cal Z}_t}}{\beta_{{\cal S}_t}}  \E^0_{{\cal Z}_t}+
\frac{\beta_{{\cal A}_t}}{\beta_{{\cal S}_t}}  \E^0_{{\cal A}_t}
\end{equation}
\noindent gives, for example,
$$
\E^0_{{\cal S}_t} \pi_{{\cal S}_t}=
\frac{\beta_{{\cal Z}_t}}{\beta_{{\cal S}_t}}
 \E^0_{{\cal Z}_t} \pi_{{\cal Z}_t} 
+
\frac{\beta_{{\cal A}_t}}{\beta_{{\cal S}_t}} 
\E^0_{{\cal A}_t} \pi_{{\cal A}_t}\text{.}$$

Notice that the mass transport principle (see \cite{LaTh09}
or Appendix \ref{app:mtpua}) implies that
\begin{eqnarray}
\beta_{{\cal Z}_t} \E^0_{{\cal Z}_t} \pi_{{\cal R}_t}  & = & 
\beta_{{\cal R}_t} \E^0_{{\cal R}_t} \pi_{{\cal Z}_t}  \\ 
\beta_{{\cal A}_t} \E^0_{{\cal A}_t} \pi_{{\cal R}_t}  & = & 
\beta_{{\cal R}_t} \E^0_{{\cal R}_t} \pi_{{\cal A}_t}  . 
\end{eqnarray}
\begin{lem} Under the foregoing assumptions,
\begin{eqnarray}
\label{eq:difz}
\frac d {dt} \beta_{{\cal Z}_t}
& = &  
-\beta_{{\cal Z}_t}
\E^0_{{\cal Z}_t} \pi_{{\cal Z}_t+{\cal R}_t}  
+\beta_{{\cal R}_t} \E^0_{{\cal R}_t} \pi_{{\cal A}_t} \\
\label{eq:difa}
\frac d {dt} \beta_{{\cal A}_t}
& = &
-\beta_{{\cal A}_t} \E^0_{{\cal A}_t} 
\pi_{{\cal A}_t+{\cal R}_t}  
+\beta_{{\cal R}_t} \E^0_{{\cal R}_t} \pi_{{\cal Z}_t} \\
\label{eq:difrplus}
\frac d {dt} \beta_{{\cal R}_t}
& = & 
\lambda 
-\beta_{{\cal R}_t} \E^0_{{\cal R}_t} \pi_{{\cal R}_t+{\cal Z}_t+{\cal A}_t}, 
\end{eqnarray}
where all the terms
found on the right hand sides of these differential equations
are uniformly bounded in $t$.
\end{lem}
\begin{proof}
By arguments similar to those of Lemma \ref{lem:boundspecpres},
both
$$\beta_{{\cal R}_t} \E^0_{{\cal R}_t} \pi_{{\cal R}_t+{\cal Z}_t} 
+\beta_{{\cal Z}_t} \E^0_{{\cal Z}_t} \pi_{{\cal R}_t+{\cal Z}_t} 
$$
and
$$\beta_{{\cal R}_t} \E^0_{{\cal R}_t} \pi_{{\cal R}_t+{\cal A}_t} 
+\beta_{{\cal A}_t} \E^0_{{\cal A}_t} \pi_{{\cal R}_t+{\cal A}_t} 
$$
are uniformly bounded in $t$, which in turn implies that all the terms
found on the right hand sides of the differential equations 
are uniformly bounded in $t$.

The death rate of a typical zombie is 
$\E^0_{{\cal Z}_t} \pi_{{\cal Z}_t + {\cal R}_t} $. 
The rate at which a regular point is
transformed into a zombie is 
$\E^0_{{\cal R}_t} \pi_{{\cal A}_t} $.
The equations are then obtained by arguments similar
to those used in Appendix \ref{appdiff} to prove \eqref{eq:difpsi'ante}.
\end{proof}

Notice that (\ref{eq:difrplus}) can be rewritten as
\begin{eqnarray}
\label{eq:difr}
\frac d {dt} \beta_{{\cal R}_t}
 = 
\lambda 
-\beta_{{\cal R}_t} \E^0_{{\cal R}_t} \pi_{{\cal R}_t} 
-\beta_{{\cal Z}_t} \E^0_{{\cal Z}_t} \pi_{{\cal R}_t}  
-\beta_{{\cal A}_t} \E^0_{{\cal A}_t} \pi_{{\cal R}_t}  .
\end{eqnarray}

These equations are consistent with those established in the last subsection.
When adding (\ref{eq:difz}) and (\ref{eq:difrplus}),
and when using the fact that 
$\beta_{{\Phi}'_t}=\beta_{{\cal Z}_t}+\beta_{{\cal R}_t}$, we get
\begin{eqnarray}
\label{eq:difpsi'}
\frac d {dt} \beta_{{\Phi}'_t} & = & 
\lambda 
-\beta_{{\cal R}_t} \E^0_{{\cal R}_t} \pi_{{\cal R}_t} 
-\beta_{{\cal Z}_t} \E^0_{{\cal Z}_t} \pi_{{\cal Z}_t} 
-2\beta_{{\cal R}_t} \E^0_{{\cal R}_t} \pi_{{\cal Z}_t} 
\nonumber \\
& = & 
\lambda
-\beta_{{\Phi}'_t} \E^0_{{\Phi}'_t} \pi_{{\Phi}'_t} ,
\end{eqnarray}
which is (\ref{eq:difpsi'ante}).

\subsubsection{Properties of Densities}

When adding (\ref{eq:difz}) and (\ref{eq:difa})
and when using the relation \eqref{eq:zas},
we get:
\begin{lem}
Under the foregoing assumptions,
\begin{eqnarray}
\label{eq:difs}
\frac d {dt} \beta_{{\cal S}_t}
&= & -\beta_{{\cal S}_t} \E^0_{{\cal S}_t}  \pi_{{\cal S}_t}.
\end{eqnarray}
\end{lem}
Hence
\begin{eqnarray}
\label{eq:soldifs}
\beta_{{\cal S}_t} =\beta_{{\cal S}_0} \exp \left(-\int_0^t
\E^0_{{\cal S}_u} \pi_{{\cal S}_u}
du\right).
\end{eqnarray}
It follows from
(\ref{eq:difs}) that $\beta_{{\cal S}_t}$ decreases and hence tends
to a limit as $t$ tends to $\infty$.

\subsection{Death Pressure}
We recall that $\Phi'_t$ denotes the point process built
by Sheriff for the initial condition ${\cal Z}_0$.
\begin{lem}
We have
\begin{equation}
\frac{d}{dt}\left( \beta_{\Phi'_t} \E^0_{\Phi'_t}
\pi_{\Phi'_t} \right)
=
2\lambda a \beta_{\Phi'_t} 
-2 \beta_{\Phi'_t}  \E^0_{\Phi'_t}
\pi^2_{\Phi'_t}.
\label{eq:sumpressure-sher-simp}
\end{equation}
\end{lem}
\begin{proof}
For all Borel sets $C$,  let 
\begin{equation}
\label{eq:defbigpi}
\Pi_{\Phi'_t} (\Phi'_t\cap C) 
=\sum_{X\in \Phi'_t\cap C} \pi_{\Phi'_t}(X),
\end{equation}
with $\pi_{\Phi'_t}(\cdot)$ defined in (\ref{eq:defpres}).
For all sets $C$, we have
\begin{equation}
\begin{split}
\frac{d}{dt} \E
\Pi_{\Phi'_t} (\Phi'_t\cap C) & =
2 \lambda a \E | \Phi'_t\cap C|\\
& \quad-\E \sum_{X\in \Phi'_t\cap C} \pi^2_{\Phi'_t}(X) \\
& \quad-\E \sum_{Y\in \Phi'_t} \pi_{\Phi'_t\cap C}(Y)\pi_{\Phi'_t}(Y) .
\end{split}
\label{eq:sumpressure-sher}
\end{equation}
The rationale is the following: $ \Pi_{\Phi'_t} (\Phi'_t\cap C) $
represents the pressure exerted by $ \Phi'_t $ on $ \Phi'_t \cap C $.
The reasons for this pressure to change with time are:
\begin{itemize}
\item
A new point can be born anywhere from the Poisson rain process.  For
each $ X \in \Phi'_{t-} \cap C $, the average pressure increase per
time unit due to arrivals is $ \lambda a $. In the case where that point is born in $ C $, which happens with intensity $ \lambda |C| $, it meets in average a pressure of strength $ \beta_{\Phi'_t} a
$, which is added to the total pressure. Using $ \beta_{\Phi'_t}  |C|=\E | \Phi'_t\cap C|$, the two effects give the first term of the R.H.S.
\item
Each $ X $ in $ \Phi'_t \cap C $ can be killed by another point.
This happens with intensity
$ \pi_{\Phi'_t}(X) $.
The death of $ X $ will decrease the total pressure
by $ \pi_{\Phi'_t}(X) $, hence the second term. This process also removes some pressure to the remaining points from $ \Phi'_t\cap C $, but this effect is considered in the third term.
\item
Each $ Y $ in $ \Phi'_t$ can be killed.
This happens with intensity
$ \pi_{\Phi'_t}(Y) $.
The death of $ Y $ will remove the pressure
$ \pi_{\Phi'_t\cap C}(Y) $ between $ Y $ and $\Phi'_t \cap C $,
hence the third term.
\end{itemize}
By standard arguments, we have
\begin{eqnarray*} 
\E \Pi_{\Phi'_t} (\Phi'_t\cap C) & = & 
\beta_{\Phi'_t}  |C|
\E^0_{\Phi'_t}
\pi_{\Phi'_t}
\end{eqnarray*}
and
\begin{eqnarray*} 
\E \sum_{X\in \Phi'_t\cap C} \pi^2_{\Phi'_t}(X)=
\beta_{\Phi'_t}  |C|
\E^0_{\Phi'_t}
\pi^2_{\Phi'_t}.
\end{eqnarray*}
Using the mass transport principle (cf Appendix \ref{app:mtpua}), we have
\begin{equation}
\label{eq:fac}
\E \sum_{Y\in \Phi'_t} \pi_{\Phi'_t\cap C}(Y)\pi_{\Phi'_t}(Y)
= \beta_{\Phi'_t}  |C|
\E^0_{\Phi'_t}
\pi^2_{\Phi'_t}.
\end{equation}
Hence (\ref{eq:sumpressure-sher}) can be rewritten as indicated in the lemma.
\end{proof}

\begin{prop}
Under the foregoing assumptions, the following
differential equations hold for the pressure of regulars
on specials:

\begin{eqnarray}
\frac{d}{dt} \left(
\beta_{{\cal Z}_t} 
\E^0_{{\cal Z}_t}
[\pi_{{\cal R}_t}]\right)
& = & \beta_{{\cal R}_t}
\E^0_{{\cal R}_t}
[\pi_{{\cal R}_t} \pi_{{\cal A}_t}]
+ \lambda \beta_{{\cal Z}_t} a\nonumber \\
& & \hspace{-3cm}
-\beta_{{\cal Z}_t}
\E^0_{{\cal Z}_t}
[\pi_{{\cal R}_t}
\pi_{{{\cal R}}_t+{{\cal Z}}_t}]
-\beta_{{\cal R}_t}
\E^0_{{\cal R}_t} 
[\pi_{{\cal Z}_t}
\pi_{{{\cal R}}_t+{{\cal Z}}_t+{{\cal A}}_t}],
\label{eq:impz}
\end{eqnarray}
as well as the symmetrical one
(i.e.\ that for $\beta_{{\cal A}_t} 
\E^0_{{\cal A}_t}
[\pi_{{\cal R}_t}]$).
In addition

\begin{eqnarray}
\frac{d}{dt} \left(
\beta_{{\cal S}_t} 
\E^0_{{\cal S}_t}
[\pi_{{\cal R}_t}]\right)
 =  
\lambda \beta_{{\cal S}_t} a
-\beta_{{\cal S}_t}
\E^0_{{\cal S}_t}
[\pi_{{\cal R}_t}
\pi_{{{\cal R}}_t+{{\cal S}}_t}]
-\beta_{{\cal R}_t}
\E^0_{{\cal R}_t} 
[ (\pi_{{{\cal S}}_t})^2 ].
\label{eq:imps}
\end{eqnarray}
For the pressure of specials on specials, we have

\begin{eqnarray}
\frac{d}{dt} \left(
\beta_{{\cal Z}_t} 
\E^0_{{\cal Z}_t}
[\pi_{{\cal Z}_t} ]\right)
 =  2 \beta_{{\cal R}_t}
\E^0_{{\cal R}_t}
[\pi_{{\cal Z}_t} \pi_{{\cal A}_t}]
- 2 \beta_{{\cal Z}_t} 
\E^0_{{\cal Z}_t}
[\pi_{{\cal Z}_t} \pi_{{\cal Z}_t+{\cal R}_t}],
\label{eq:pss}
\end{eqnarray}
as well as the symmetrical one
(i.e. that for $\beta_{{\cal A}_t} 
\E^0_{{\cal A}_t}
[\pi_{{\cal A}_t}]$).
In addition

\begin{eqnarray}
\label{eq:anteSS}
\frac{d}{dt} \left(
\beta_{{\cal S}_t} 
\E^0_{{\cal S}_t}
[\pi_{{\cal S}_t} ]\right)
 =  4 \beta_{{\cal R}_t} 
\E^0_{{\cal R}_t}
[\pi_{{\cal Z}_t} \pi_{{\cal A}_t}]
- 2 \beta_{{\cal S}_t} 
\E^0_{{\cal S}_t}
[\pi_{{\cal S}_t} \pi_{{\cal S}_t+{\cal R}_t}].
\end{eqnarray}
Finally, for the pressure of regulars on regulars, we have

\begin{eqnarray}
\label{eq:prr}
\frac{d}{dt} \left(
\beta_{{\cal R}_t} 
\E^0_{{\cal R}_t}
[\pi_{{\cal R}_t} ]\right)
 =  \lambda \beta_{{\cal R}_t}  a
- 2 \beta_{{\cal R}_t} 
\E^0_{{\cal R}_t}
[\pi_{{\cal R}_t} \pi_{{\cal R}_t+{\cal S}_t}].
\end{eqnarray}
All terms in the RHSs of these differential equations are
uniformy bounded.
\end{prop}
\begin{proof}
The last property is obtained by the same arguments as for densities
(see Section \ref{sec:tight}).

For all Borel sets $C_1,C_2$  
and for all point processes ${\cal X}_t, {\cal Y}_t$, consider
\begin{equation}
\Pi_{{\cal Y}_t\cap C_2} ({\cal X}_t\cap C_1) 
=\sum_{X\in {\cal X}_t\cap C_1} \pi_{{\cal Y}_t\cap C_2}(X)\text{.}
\end{equation}
We remind the pressure analogy used in the proof of \eqref{eq:sumpressure-sher-simp}:
 $\Pi_{{\cal Y}_t\cap C_2} ({\cal X}_t\cap C_1)$
can be seen as the death pressure exerted by $ {\cal Y}_t\cap C_2 $
on $ {\cal X}_t\cap C_1 $. Note that because of the symmetry of the processes,
it is also the pressure exerted by $ {\cal X}_t\cap C_1 $ on $ {\cal Y}_t\cap C_2 $.

We first prove (\ref{eq:impz}).
For all sets $C$, we have
\begin{equation}
   \begin{split}
\frac{d}{dt} \E
\Pi_{{\cal R}_t} ({\cal Z}_t\cap C) &=
\E \sum_{X\in {\cal R}_t\cap C}
\pi_{{\cal R}_t}(X) \pi_{{\cal A}_t}(X)\\
&\quad + \lambda a \E | {\cal Z}_t\cap C|\\
&\quad -\E \sum_{X\in {\cal Z}_t\cap C}
\pi_{{\cal R}_t}(X) \pi_{{{\cal R}_t}+{{\cal Z}}_t}(X)\\
&\quad -\E \sum_{X\in {\cal R}_t} \pi_{{\cal Z}_t\cap C}(X) \pi_{{{\cal R}_t}+{{\cal Z}}_t+{{\cal A}}_t}(X).
   \end{split}
\label{eq:sumpressure}
\end{equation}
The reason is the following: $ \Pi_{{\cal R}_t} ({\cal Z}_t\cap C) $
represents the pressure exerted by $ {\cal R}_t $ on $ {\cal Z}_t \cap C $.
The reasons for this pressure to change with time are:
\begin{itemize}
\item A regular point $ X\in {\cal R}_t \cap C $ can be turned into a new zombie  due
to a regular--antizombie interaction. For each $X\in {\cal R}_t \cap C$,
this happens with intensity $ \pi_{{\cal A}_t}(X) $. 
The newborn zombie will experience a pressure $ \pi_{{\cal R}_t}(X) $
(we recall the convention $ f(0)=0 $);
hence the first term in \eqref{eq:sumpressure}.
This process also removes $ X $ as a regular point, but this effect is
considered in the fourth term.
\item A new regular point can be born from the Poisson rain process.
For each $ Z \in {\cal Z}_t \cap C $, the average pressure
increase per time unit due to arrivals is $ \lambda a $,
hence the second term.
\item Each zombie $X\in {\cal Z}_t \cap C $ can be killed by a regular point or a zombie. This happens with intensity
$ \pi_{{\cal R}_t+{\cal Z}_t}(X) $.
The death of $ X $ will decrease the total pressure
by $ \pi_{{\cal R}_t}(X) $, hence the third term.
\item Each regular point $ X\in {\cal R}_t$ can be killed by
anyone (regular or special). This happens with intensity
$ \pi_{{\cal R}_t+{\cal Z}_t+{\cal A}_t}(X) $.
The death of $ X $ will remove the pressure
$ \pi_{{\cal Z}_t\cap C}(X) $ between $ X $ and ${\cal Z}_t \cap C $,
hence the last term.
\end{itemize}
Each term in (\ref{eq:sumpressure}) including the one differentiated are
uniformly bounded for the same reasons 
as those used
 in the proof of Lemma \ref{lem:boundspecpres} (tightness of
the quantities of interest for both initial conditions).
From the very definition of Palm probability, we can rewrite the 
term which is differentiated in (\ref{eq:sumpressure}) as
\begin{eqnarray*} 
\E \Pi_{{\cal R}_t} ({\cal Z}_t\cap C) & = & 
\beta_{{\cal Z}_t}  |C|
\E^0_{{\cal Z}_t}
[\pi_{{\cal R}_t}]
\end{eqnarray*}
and the first and third terms on the R.H.S. as
\begin{eqnarray*} 
\E \sum_{X\in {\cal R}_t\cap C}
\pi_{{\cal R}_t}(X) \pi_{{\cal A}_t}(X)
& = & \beta_{{\cal R}_t}  |C|
\E^0_{{\cal R}_t}
[\pi_{{\cal R}_t} \pi_{{{\cal A}_t}}]
\\
\E \sum_{X\in {\cal Z}_t\cap C}
\pi_{{\cal R}_t}(X) \pi_{{{\cal R}_t}+{{\cal Z}}_t}(X)
& = & \beta_{{\cal Z}_t}  |C|
\E^0_{{\cal Z}_t}
[\pi_{{\cal R}_t} \pi_{{{\cal R}_t}+{{\cal Z}}_t}],
\end{eqnarray*}
respectively.
In addition, we show in Appendix \ref{app:mtpua} that the following
identity holds for the fourth term:
\begin{eqnarray}
\E \sum_{X\in {\cal R}_t} \pi_{{\cal Z}_t\cap C}(X) \pi_{{{\cal R}_t}+{{\cal Z}}_t+{{\cal A}}_t}(X)
= \E \sum_{X\in {\cal R}_t\cap C}
\pi_{{{\cal Z}}_t} (X)
\pi_{{{\cal R}}_t+{{\cal Z}}_t+{{\cal A}}_t}(X).
\label{eq:mtpua}
\end{eqnarray}
Hence
\begin{equation}
\E \sum_{X\in {\cal R}_t} \pi_{{\cal Z}_t\cap C}(X) \pi_{{{\cal R}_t}+{{\cal Z}}_t+{{\cal A}}_t}(X)
= \beta_{{\cal R}_t} |C| \E^0_{{\cal R}_t} \left[ \pi_{{\cal Z}_t} \pi_{{{\cal R}_t}+{{\cal Z}}_t+{{\cal A}}_t}\right].
\end{equation}
We get (\ref{eq:impz}) when dividing (\ref{eq:sumpressure}) by $|C|$.
The other equations are obtained in the same way.
\end{proof}

Here are a few observations on these equations.
Consider e.g. (\ref{eq:imps}).
The only positive term in the RHS of this equation is $ \lambda \beta_{{\cal S}_t} a $. In particular, the positive term
$  \beta_{{\cal R}_t} \E^0_{{\cal R}_t} [\pi_{{\cal R}_t} \pi_{{\cal S}_t}] $ 
(contamination of an $ R $ by an $ S $) is nullified.
The reason is obvious if one considers the death of a typical
$ R\in {\cal R}_t$. $ R $ undergoes a pressure $ \pi_{{\cal R}_t}(R) $ 
from $ {\cal R}_t $ and $ \pi_{{\cal S}_t}(R) $
from $ {\cal S}_t $. If the killing comes from the pressure
of ${\cal S}_t $, a new special is created in $ R $
and the pressure from ${\cal R}_t $ will be \emph{added} 
to the pressure between ${\cal S}_t $ and ${\cal R}_t
$\footnote{Meanwhile, the pressure $ \pi_{{\cal S}_t}(R) $ is also
removed, hence the $ -(\pi_{{\cal S}_t}(R))^2 $ term.}.
Conversely, if the killing comes from the pressure of
${\cal R}_t $, $ R $ is removed and its pressure
from ${\cal S}_t $ is \emph{subtracted}.

It is not difficult to check the following consistency property:
when adding 
twice (\ref{eq:impz}), (\ref{eq:pss}) and (\ref{eq:prr}),
we get back (\ref{eq:sumpressure-sher-simp}) as expected.
\begin{prop}
\label{lem35+37}
Under the foregoing assumptions, both $\E^0_{{\cal S}_t} \pi_{{\cal R}_t}$ 
and $\E^0_{{\cal S}_t} \pi_{{\cal S}_t}$ are uniformly bounded with respect to $ t $.
\end{prop}
\begin{proof}
By adding twice (\ref{eq:imps}) and (\ref{eq:anteSS}), we get
\begin{eqnarray*}
\frac{d}{dt} \left(
\beta_{{\cal S}_t} 
\E^0_{{\cal S}_t}
[2\pi_{{\cal R}_t}+\pi_{{\cal S}_t}]\right)
& = & 
2\lambda \beta_{{\cal S}_t} a
-2\beta_{{\cal S}_t} \E^0_{{\cal S}_t} [(\pi_{{\cal R}_t}+\pi_{{\cal S}_t})^2] 
\\  & &
-2 \beta_{{\cal R}_t} \E^0_{{\cal R}_t} [ (\pi_{{{\cal A}}_t}+\pi_{{{\cal Z}}_t})^2 ]
+4 \beta_{{\cal R}_t} \E^0_{{\cal R}_t}
[\pi_{{\cal Z}_t} \pi_{{\cal A}_t}] \\
& \le &
2\lambda \beta_{{\cal S}_t} a 
-2\beta_{{\cal S}_t} \E^0_{{\cal S}_t} [(\pi_{{\cal R}_t}+\pi_{{\cal S}_t})^2].
\end{eqnarray*}
By making use of (\ref{eq:difs}) in the last equation, we get 
\begin{eqnarray*}
\frac{d}{dt} 
\E^0_{{\cal S}_t}
[2 \pi_{{\cal R}_t} +  \pi_{{\cal S}_t}]
& \le & 2 \lambda a +
\E^0_{{\cal S}_t} [\pi_{{\cal S}_t}]
\E^0_{{\cal S}_t} [2 \pi_{{\cal R}_t} +  \pi_{{\cal S}_t}]
-2 \E^0_{{\cal S}_t} [(\pi_{{\cal R}_t}+\pi_{{\cal S}_t})^2]\\
& \le & 2 \lambda a +
\left(\E^0_{{\cal S}_t} [\pi_{{\cal S}_t}+\pi_{{\cal R}_t}]\right)^2
-2 \E^0_{{\cal S}_t} [(\pi_{{\cal R}_t}+\pi_{{\cal S}_t})^2]
\\ & \le & 2 \lambda a 
- (\E^0_{{\cal S}_t} [\pi_{{\cal R}_t}+\pi_{{\cal S}_t}])^2
\end{eqnarray*}
Thus $\E^0_{{\cal S}_t}
[2 \pi_{{\cal R}_t} +  \pi_{{\cal S}_t}]$
is decreasing whenever
$\E^0_{{\cal S}_t} [\pi_{{\cal R}_t}+  \pi_{{\cal S}_t}] > \sqrt{2\lambda a}$. As $\E^0_{{\cal S}_t} [2\pi_{{\cal R}_t}+  \pi_{{\cal S}_t}] > 2\sqrt{2\lambda a}$ implies $\E^0_{{\cal S}_t} [\pi_{{\cal R}_t}+  \pi_{{\cal S}_t}] > \sqrt{2\lambda a}$, we get
$$\limsup_{t\to \infty} \E^0_{{\cal S}_t}
[2 \pi_{{\cal R}_t} +  \pi_{{\cal S}_t}]
  \le 2 \sqrt{2\lambda a}.
$$
\end{proof}

\section{Construction of the Stationary Regime} \label{sec:ergo}

The most important result of Section \ref{sec:init} for the
construction of the stationary regime is Theorem \ref{zshclaimherfin}, which
shows that the perturbation induced
by any point (in a finite intensity augmentation point process) a.s.\ vanishes
with time. If we were on a finite domain, there would
be a finite number of additional points, each with an influence
that vanishes in finite time and the processes with and without
augmentation would hence coincide after a finite time.
This would provide a natural way of proving the existence and the
uniqueness of the stationary regimes through a coupling
from the past construction.
This is the line of thought that we follow. The main technical 
difficulty consists in proving that the influence of the
additional points vanishes fast enough. 

In this section, we assume that the initial condition is
a thinned homogeneous Poisson point process of finite intensity.

\subsection{Exponential Decay of the Density of Special Points}
In this subsection, we suppose that Assumptions 0--3 hold.
We will use the following notation for $\sigma$-algebras:
\begin{eqnarray*}
\mathcal{D}_t&=&\sigma(\mathcal{Z}_0\cup\Psi_{(0,t]})\\
\mathcal{T}_t&=&\sigma(T_{pq}:\,p,q\in\mathcal{Z}_0\cup\Psi_{(0,t]})\\
\mathcal{I}_t&=&\sigma(I_{pq}:\,p,q\in\mathcal{Z}_0\cup\Psi_{(0,t]})\\
\mathcal{G}_t&=&\mathcal{D}_t\vee \mathcal{T}_t\vee \mathcal{I}_t.
\end{eqnarray*}
Note that $\widetilde \Phi_t$, introduced in Section \ref{sum1}, is $\mathcal{G}_t$-measurable.

For any special node $z\in\widetilde \Phi_t$,
denote by $M_{z,t,s}$ the number of
points that are offsprings of $z$
and that are still alive at time $s\ge t$
(here when $z$ kills an ordinary point, the latter becomes
a first generation offspring of $z$; when this node
kills another ordinary point, the latter is seen
as a second generation offspring of $z$, etc).
The following lemma is a direct corollary of
what was already established in the proofs of
Theorems \ref{zombiesheriffthm} and \ref{zshclaimherfin}:
\begin{lem}
\label{meanzfamlemma}
Under the assumptions of Theorem \ref{zombiesheriffthm},
for all special nodes $z\in\widetilde \Phi_t$ and all $s\ge t$,
$$
\cE{M_{z,t,s}}{\mathcal{G}_t\vee\mathcal{D}_s\vee\mathcal{T}_s}
\le1.
$$
\end{lem}

For any $t>0$, any special point $z\in\widetilde \Phi_t$
and any fixed positive numbers
$r$ and $\epsilon$, consider the event
\begin{eqnarray*}
A_1 & = &
\set{\Psi_{(t,t+\epsilon)} (B(z,r))=2 }.
\end{eqnarray*}
On $A_1$, denote by $(a,t_a),(b,t_b)$ the random locations 
of the two points of $\Psi_{(t,t+\epsilon)} \cap B(z,r)$ and
consider the events
\begin{eqnarray*}
A_2 & = &\hspace{-.3cm}
\set{ t_a\vee t_b<T_{za}<T_{zb}<T_{ab}<t+\epsilon},\\
A_3 & = & \hspace{-.3cm}
\set{
\forall p\in\set{z,a,b},
\forall q\in \widetilde \Phi_t\cup e_1(\Psi_{(t,t+\epsilon)})\cap B(z,r)^c:\,
T_{pq}>t+\epsilon },
\end{eqnarray*}
where $e_1(\Psi)$ denotes the projection $\Re^d\times\Re\to\Re^d$.
Let $G=A_1\cap A_2\cap A_3$. Notice that $G$ is in
$\mathcal{G}_t\vee\mathcal{D}_{t+\epsilon}\vee\mathcal{T}_{t+\epsilon}$.
\begin{lem}
\label{downzdriftlemma}
Under Assumptions 0--3,
for all $t,\epsilon>0$, 
and for all special points $z\in\widetilde \Phi_t$,
\begin{eqnarray*}
\cE{M_{z,t,t+\epsilon}}{\mathcal{G}_t} & \le & 1-\frac14
\P [G\mid \mathcal{G}_t].
\end{eqnarray*}
\end{lem}
\begin{proof}
On $G$, the value of
$M_{z,t,t+\epsilon}$ depends only on the random variables $I_{za}$,
$I_{zb}$ and $I_{ab}$.
If $I_{za}=-1$, $z$ is killed
without producing any offspring after $t$. If $I_{za}=1$ and
$I_{zb}=-1$, first $a$ becomes special and then $b$ kills $z$. If then
$I_{ab}=1$, both $a$ and $b$ become special, whereas in the opposite
case $z$'s family dies out. Finally, if $I_{za}=I_{zb}=1$, both $a$
and $b$ become special of same kind and, at time $T_{ab}$, one of them
kills the other. Thus the expected value of $M_{z,t,t+\epsilon}$ on $G$ is
$
(4\cdot0+1\cdot 2+1\cdots 0+2\cdot 2)/8=\frac34.
$ 
Hence
\begin{eqnarray*}
\cE{\indic{G}M_{z,t,t+\epsilon}}{\mathcal{G}_t}
& = & \frac34 \P[G\mid \mathcal{G}_t].
\end{eqnarray*}
Similarly,
\begin{eqnarray*}
\cE{\indic{G^c}M_{z,t,t+\epsilon}}{\mathcal{G}_t} 
& = &
\cE{\indic{G^c} 
\cE{M_{z,t,t+\epsilon}}
{\mathcal{D}_{t+\epsilon}\vee\mathcal{T}_{t+\epsilon} \vee \mathcal{G}_t} }{\mathcal{G}_t}\\
&\le &
\cE{\indic{G^c}}{\mathcal{G}_t}, 
\end{eqnarray*}
where we used Lemma \ref{meanzfamlemma}.
Hence
\begin{eqnarray*}
\cE{M_{z,t,t+\epsilon}}{\mathcal{G}_t}  & = &
\cE{\indic{G}M_{z,t,t+\epsilon}}{\mathcal{G}_t} +
\cE{\indic{G^c}M_{z,t,t+\epsilon}}{\mathcal{G}_t}  \\
& = &
\frac34 
\P[G\mid \mathcal{G}_t] 
+\cE{\indic{G^c}M_{z,t,t+\epsilon}}{\mathcal{G}_t}\\
& \le &
\frac34 \P[G\mid \mathcal{G}_t] +1- \P[G\mid \mathcal{G}_t] .
\end{eqnarray*}
\end{proof}
From classical properties
of Poisson point processes, reminding that $ \nu_d $ denotes the volume of a unit ball in $\Re^d$, we get
\begin{eqnarray*}
\P [G\mid \mathcal{G}_t] & = & \P[A_1]\
\frac {1 } { (\nu_dr^d \epsilon)^2 }
\int\limits_{B(z,r)} 
\int\limits_{B(z,r)}
\int\limits_{[t,t+\epsilon]}
\int\limits_{[t,t+\epsilon]}
\\
&& \hspace{-2cm}
\P[u\vee v<T_{zx}<T_{zy}<T_{xy}<t+\epsilon] \times
\\
&& \hspace{-2cm}
\cPr{\forall p\in\set{z,x,y},
\forall q\in \widetilde \Phi_t\cup e_1(\Psi_{(t,t+\epsilon)})\cap B(z,r)^c:\,
T_{pq}>t+\epsilon}{\mathcal{G}_t}\\ &&
\hspace{6cm}dx dy du dv .
\end{eqnarray*}
Notice that $\P[A_1]$ is a constant that does not depend on $t$.
Similarly, if $r$ is such that $f(2r)>0$, then
$\P[u\vee v<T_{zx}<T_{zy}<T_{xy}<t+\epsilon]$ is bounded from below
by a constant that does not depend on $t,u,v,x,y$. 
From this and the independence properties of the Poisson
rain after $t$ and $\mathcal{G}_t$, we get that
for all $r$ as above, there exists a 
constant $0<F(r,\epsilon)<1$ such that
\begin{eqnarray}
\P[G\mid {\cal F}_t] & \ge & F(r,\epsilon)
\frac {1 } { (\nu_dr^d)^2 }
\int\limits_{B(z,r)} 
\int\limits_{B(z,r)} \nonumber \\
& & \hspace{-2cm}
\cPr{\forall p\in\set{z,x,y},
\forall q\in \widetilde \Phi_t:\,
T'_{pq}>t+\epsilon}{\mathcal{G}_t}
dx dy ,
\label{eq:amelior}
\end{eqnarray}
where the random variables $T'_{pq}$ are mutually
independent and $T'_{pq}$ is $t$ plus an exponential of
parameter $2f(\|p-q\|)$.

\begin{thm}
\label{zombintexpdownlemma}
Under Assumptions 0--3,
there exists an $\alpha>0$ such that
\begin{equation}
\beta_{\mathcal{S}_t}\le e^{-\alpha t}
\label{eq:decexp}
\end{equation}
for $t$ large enough.
\end{thm}
\begin{proof}

From Proposition \ref{lem35+37}, there exists a $J<\infty$ such that
\begin{equation}
\E^0_{\mathcal{S}_t} [
\pi_{\mathcal{S}_t} (0) + \pi_{\mathcal{R}_t} (0)]<J,
\label{eq:HH}
\end{equation}
uniformly in $t$.
Since the 
point process $\mathcal{S}_t$ is spatially stationary, we can use
the Campbell--Mecke formula to prove that for all $\epsilon >0$,
$$\beta_{\mathcal{S}_{t+\epsilon}}=
\beta_{\mathcal{S}_t} \E^0_{\mathcal{S}_t} [M_{0,t,t+\epsilon}]$$
with $M_{0,t,s}$ the number of special points
offspring of the origin living at time $s$, under $\P^0_{\mathcal{S}_t}$.
From Lemma \ref{downzdriftlemma} and (\ref{eq:amelior}), we get
$$\beta_{\mathcal{S}_{t+\epsilon}}\le
\beta_{\mathcal{S}_t} -\beta_{\mathcal{S}_t}\frac 1 4 F(r,\epsilon) \xi_t
$$
with
\begin{eqnarray*}
\xi_t =
\frac {1 } {(\nu_dr^d)^2}
\int\limits_{B(0,r)} 
\int\limits_{B(0,r)}
\P^0_{\mathcal{S}_t} \left[\forall p\in\set{0,x,y},
\forall q\in \widetilde \Phi_t^0:\,
T'_{pq}>t+\epsilon\right]
dx dy.
\end{eqnarray*}
Let us show that when (\ref{eq:HH}) holds, there exists an $r$, a $\epsilon$
and constant $0<C(r,\epsilon)\le 1$ that does not depend on $t$
and such that 
\begin{eqnarray}
\label{eq:ubound}
\xi_t> C(r,\epsilon).
\end{eqnarray}

We first explain the idea of the proof
of (\ref{eq:ubound}) by ignoring the conditions on $x$ and $y$.
\begin{eqnarray*}
\P^0_{\mathcal{S}_t} [ T'_{0q} > t+ \epsilon, \ \forall q\ne 0
\in \widetilde \Phi_t^0]
&= &
\E^0_{\mathcal{S}_t} \left[
\P^0_{\mathcal{S}_t}[ T_{0q} > \epsilon,\ \forall q\ne 0
\in\widetilde  \Phi_t^0 \mid  \widetilde \Phi_t ]\right]\\
&= &
\E^0_{\mathcal{S}_t} \left[
\prod_{q\ne 0 \in \widetilde \Phi_t^0} e^{-2 \epsilon f(\|q\|)}
 \right]\\
&\ge & 1 - \epsilon
\E^0_{\mathcal{S}_t} \left[
\sum_{q\ne 0 \in \widetilde \Phi_t^0} 2 f(\|q\|)\right]\\
&= & 1 - 2 \epsilon
\left(\E^0_{\mathcal{S}_t} \pi_{\mathcal{S}_t}(0)
+\E^0_{\mathcal{S}_t} \pi_{\mathcal{R}_t}(0)
\right)\\
&\ge & 1 -2\epsilon J.
\end{eqnarray*}

We now consider $x$ and $y$ in addition to 0. By the 
same arguments, we have
\begin{eqnarray}
 \P^0_{\mathcal{S}_t} \left[\forall p\in\set{0,x,y},
\forall q\in \widetilde \Phi_t^0:\,
T'_{pq}>t+\epsilon\right] & & \nonumber \\
&  & \hspace{-4cm}
= \E^0_{\mathcal{S}_t} 
\left[
\prod_{q\ne 0 \in \widetilde \Phi_t^0}
e^{-2 \epsilon (f(\|q\|)+ f(\|q-x\|)+ f(\|q-y\|)}
\right]\nonumber
\\ & & \hspace{-4cm}
\ge  1 - 2 \epsilon
\left(
\E^0_{\mathcal{S}_t} \pi_{\mathcal{S}_t}(0)
+\E^0_{\mathcal{S}_t} \pi_{\mathcal{S}_t}(x)
+\E^0_{\mathcal{S}_t} \pi_{\mathcal{S}_t}(y)
\right. \nonumber
\\ & & \hspace{-3.5cm} 
\left. 
+\E^0_{\mathcal{S}_t} \pi_{\mathcal{R}_t}(0)
+\E^0_{\mathcal{S}_t} \pi_{\mathcal{R}_t}(x)
+\E^0_{\mathcal{S}_t} \pi_{\mathcal{R}_t}(y)
\right).
\label{eq:partout}
\end{eqnarray}

Let us now show that, under the foregoing assumptions,
if $\E^0_{\mathcal{S}_t} \pi_{\mathcal{S}_t}(0)$
is uniformly bounded, then so are
$\E^0_{\mathcal{S}_t} \pi_{\mathcal{S}_t}(x)$
and $\E^0_{\mathcal{S}_t} \pi_{\mathcal{S}_t}(y)$.
The initial condition satisfies the assumptions
of Section \ref{generalinit} and arrivals in $(0,t)$
form a marked Poisson point process on $\R^d$. Both 
are motion-invariant.
Motion invariance is preserved by the dynamics. Hence we have
$$\E^0_{\mathcal{S}_t} \pi_{\mathcal{S}_t}(0) = \nu_d d \frac 1
{\beta_{{\cal S}_t}} \int_{r>0} f(r) \rho^{[2]}_{{\cal S}_t}(r) r^{d-1} dr,$$
where $\rho^{[2]}_{{\cal S}_t}(r)$ is the radial component of the (motion
invariant) density of the reduced second moment measure of ${\cal
  S}_t$ (see Section \ref{ss:be} for definitions).  The fact that the
last function is uniformly bounded implies that $\E^0_{\mathcal{S}_t}
{\cal S}_t(B(0,b))$ is uniformly bounded for all $b$ such that
$f(b) >0$.  It also implies that for all $H>1$,
$$
\nu_d d \frac 1 {\beta_{{\cal S}_t}}
\int_{r>0} f(\frac r H) \rho^{[2]}_{{\cal S}_t}(\frac r H) r^{d-1} dr$$
is uniformly bounded.
For all $x$ with $\|x\|< b$, with $b$ such that $f(b) >0$,
from monotonicity and boundedness, we have
\begin{eqnarray*}
\E^0_{\mathcal{S}_t} \pi_{\mathcal{S}_t}(x) 
& \le &
K \E^0_{\mathcal{S}_t} [{\mathcal{S}_t}(B(0,b))] 
+ \nu_d d \frac 1 {\beta_{{\cal S}_t}}
\int_{r\ge \rho}  f(\frac r H) \rho^{[2]}_{{\cal S}_t}(\frac r H) r^{d-1} dr,
\end{eqnarray*}
with 
$ H= \frac b {b -\|x\|},$
which shows that $\E^0_{\mathcal{S}_t} \pi_{\mathcal{S}_t}(x)$
is uniformly bounded.
By similar arguments, since $\E^0_{\mathcal{S}_t} \pi_{\mathcal{R}_t}(0)$ 
is uniformly bounded, then so are
$\E^0_{\mathcal{S}_t} \pi_{\mathcal{R}_t}(x)$
and $\E^0_{\mathcal{S}_t} \pi_{\mathcal{R}_t}(y)$. 

Thus, thanks to the uniform boundedness 
of all 6 terms showing up
in (\ref{eq:partout}), one can choose a $\epsilon$ and a $r$ small enough
in (\ref{eq:partout}) for (\ref{eq:ubound}) to hold.

Hence, for all $t$,
$$ \beta_{\mathcal{S}_{t+\epsilon}}\le \beta_{\mathcal{S}_t}
-\beta_{\mathcal{S}_t} \gamma,
$$
with $0< \gamma=\gamma(r,\epsilon)= \frac 14 F(r,\epsilon)
C(r,\epsilon) <1$.

Since the function $t\to \beta_{\mathcal{S}_t}$ is monotone non-increasing,
for $\epsilon>0$ as defined above,
\begin{eqnarray*}
\beta_{\mathcal{S}_t}  \le  \beta_{\mathcal{S}_{\epsilon \lfloor \frac t \epsilon \rfloor}} 
\le  \beta_{\mathcal{S}_0} (1-\gamma)^{\lfloor \frac t \epsilon \rfloor},
\end{eqnarray*}
with the last inequality following from the above bound.
Since $\lfloor \frac t \epsilon \rfloor\ge -1 + \frac t \epsilon$, it follows that
\begin{eqnarray*}
\beta_{\mathcal{S}_t}  \le  
\frac{\beta_{\mathcal{S}_0}}{1-\gamma} (1-\gamma)^{\frac t \epsilon},
\end{eqnarray*}
for all $t$, which concludes the proof.
\end{proof}

\begin{thm}
\label{thm:compactcoupling}
Consider two executions of
Sheriff$^Z$: that with an empty initial condition and that with
a stationary and ergodic initial point process
$\mathcal{Z}_0$ which satisfies the
conditions of Subsection \ref{generalinit}.
Under Assumptions 0--3,
for all compacts $C$ of $\R^d$, there exists a random
time $\tau(C)$ with finite expectation such that for all $t\ge \tau(C)$,
these two executions coincide in $C$.
\end{thm}

\begin{proof}
Denote by $N_t=\mathcal{S}_t(C)$ the number of special points living in $C$ at time $t$. It suffices to show that the random time
$$
\tau(C)=\sup\eset{t\ge0}{N_t>0}
$$
has finite expectation. Note first that Theorem \ref{zombintexpdownlemma} already yields
\begin{equation}
\label{speclifefin}
\E{\int_0^\infty N_t\D{t}}=\int_0^\infty \E{N_t}\D{t}=|C|\int_0^\infty\beta_{\mathcal{S}_t}\D{t}<\infty.
\end{equation}
Write
$$
N_t=N_0+N^+_t-N^-_t,
$$
where $N^+$ and $N^-$ are the counting processes of births and deaths of special points in $C$. 
Since the stochastic intensity of $N^+$, say $\lambda^{N^+}_t$, is
$$
\lambda^{N^+}_t=\sum_{x\in\mathcal{R}_t\cap C}\pi_{\mathcal{S}_t}(x),
$$
we have, using the mass transport principle,
\begin{equation}
\label{expspecbirthin}
\E{\lambda^{N^+}_t}=\beta_{\mathcal{R}_t}|C| \E^0_{\mathcal{R}_t}\pi_{\mathcal{S}_t}
=\beta_{\mathcal{S}_t}|C| \E^0_{\mathcal{S}_t}\pi_{\mathcal{R}_t}.
\end{equation}
Denote the sorted birth and death times of special points in $C$
by $(T^{N^+}_n)_{n\ge1}$ and $(T^{N^-}_n)_{n\ge1}$, respectively.
By the definition of stochastic intensity, (\ref{expspecbirthin}), Proposition \ref{lem35+37} and Theorem \ref{zombintexpdownlemma},
\begin{eqnarray}
\label{specbirthexpfin}
\E{\sum_{n=1}^\infty T^{N^+}_n}&=&\E{\int_0^\infty t\D{N^+_t}\\
\nonumber
&=&\E{\int_0^\infty t\lambda^{N^+}_t\D{t}}\\
\nonumber
&=&\int_0^\infty t\E{\lambda^{N^+}_t}\D{t}\\
\nonumber
&\le&|C|(\sup_t\E^0_{\mathcal{S}_t}\pi_{\mathcal{R}_t})\int_0^\infty t\beta_{\mathcal{S}_t}\D{t}<\infty.
\end{eqnarray}
We can write
$$
\int_0^\infty N_t\D{t}=\sum_{n=1}^\infty T^{N^-}_n-\sum_{n=1}^\infty T^{N^+}_n,
$$
since the last sum is finite by (\ref{specbirthexpfin}).
Now the claim follows by noting that
$$
\E{\tau(C)}\le \E \sum_{n=1}^\infty T^{N^-}_n}=\E{\int_0^\infty N_t\D{t}}+\E{\sum_{n=1}^\infty T^{N^+}_n}<\infty.
$$
\end{proof}

\subsection{Coupling from the Past}

Throughout this section Assumptions 0--3 are supposed to hold
and $\mathcal{Z}_0$ is a translation invariant initial condition
which satisfies the properties listed in Subsection \ref{generalinit}.

For each location $y\in \R^d$, let $V_y$ denote the time it takes for the
point process generated by Sheriff acting on $\Psi_{(0,\infty)}$ with
the empty initial condition, and that generated by Sheriff acting on
$\Psi_{(0,\infty)}$ with the initial condition $\mathcal{Z}_0$ to
couple (i.e.\ to be identical forever) in the unit ball centered at $y$.
Under the foregoing assumptions, the random field $V_y$ is
translation invariant.
From Theorem \ref{thm:compactcoupling},
for all $y$, $\E V_y = \E V_0 < \infty$.

Now, in Theorem \ref{thm:compactcoupling}, choose
$$ \mathcal{Z}_0:=\mbox{ the nodes of $\Psi_{(-1,0]}$ alive after running
Sheriff on them,} $$
as augmentation of the initial condition of Sheriff acting on
$\Psi_{(0,\infty)}$. Let $V^{(0)}_y=V_y$ denote
the associated coupling time field.
Consider also the set of nodes of $\Psi_{(-2,-1]}$ still
alive at time -1 as augmentation of the initial condition of
Sheriff acting on $\Psi_{(-1,\infty)}$, and denote by $V^{(1)}_y$
the associated coupling time field.
The random fields $V^{(0)}_y$ and $V^{(1)}_y+1$ are stochastically equivalent:
denoting by $\theta_{t}$ the measure preserving time shift
$$ \theta_t \Psi(C\times H)=\Psi(C\times (t+H)),$$
for all Borel sets $C$ of $\R^d$ and all Borel sets $H$ of $\R$,
we get that
$$ V^{(0)}_y \circ \theta_{-1} = V^{(1)}_y +1 ,\quad \forall y.$$
Continuing like this, we obtain a sequence $(V^{(n)}_y+n)_{n\in\Nat}$
of identically distributed, and spatially stationary fields.

\begin{lem}
\label{uniquenessprop}
Under Assumptions 0--3, for all $y\in \R^d$, $V^{(n)}_y\to -\infty$ a.s.
when $n\to \infty$.
\end{lem}
\begin{proof}
Rewrite $V^{(n)}_y$ as
$V^{(n)}_y+n -n$. Since the sequence $V^{(n)}_y+n$ is stationary
and ergodic with finite mean, $\frac {V^{(n)}_y+n} n \to 0$
when $n\to \infty$. This implies the announced result.
\end{proof}

The following theorem builds a {\em time stationary} family of point processes,
compatible with the birth and death dynamics.
\begin{thm}
\label{thm:xx}
For $t\in \R$, let $ \Phi_{(-n,t)}^\emptyset(t)$
denote the point process of nodes that Sheriff builds
alive at time $t$, when starting
the dynamics at time $-n$ and with an empty initial condition.
For all $t$, the a.s. limit
\begin{equation}
\label{eq:xidef}
\Upsilon_t= \lim\limits_{n\to\infty}
\Phi_{(-n,t)}^\emptyset(t)
\end{equation}
exists and forms a time-stationary family of translation invariant
point processes on $\R^d$.
\end{thm}
\begin{proof}
From Lemma \ref{uniquenessprop},
for all compacts $C$ of $\R^d$, for all $t$,
when $n$ tends to $\infty$,
$\Phi_{(-n,t)}^\emptyset(t)$,
couples with a finite random variable $\Upsilon_t$
for $n$ larger than a finite random threshold, denoted by $\tau_t(C)$,
so that the limit (\ref{eq:xidef}) a.s. exists indeed.
The property that $(\Upsilon_t)_{t\in\Re}$
is a time-stationary family of point processes
follows from the fact that for all $t$,
$\Upsilon_t=\Upsilon_0\circ \theta_t$.
\end{proof}

The following theorem completes the analysis of the stationary
regimes in terms of convergence in distribution to $\Upsilon_0$.
We recall that a sequence of point processes $\phi_n$
converges in distribution to the point process $\phi$ if and only if for
any function $h:\R^d\to \R^+$, which has bounded support and is
continuous, $\int h d \phi_n$ converges in distribution to
$\int h d \phi$ in $\R^+$ \cite{Ka76}.

\begin{thm}
\label{thm:wholephi}
Let $\Phi_{[0,n)}^{{\cal Z}_0}(n)$
be the point process
of nodes living at time $n$ constructed
by Sheriff when starting at time 0 with some initial condition ${\cal Z}_0$.
Under Assumptions 0--3,
for all initial conditions ${\cal Z}_0$ satisfying the assumptions of
Subsection \ref{generalinit}, $\Phi_{[0,n)}^{{\cal Z}_0}(n)$ converges
in distribution to $\Upsilon_0$.
\end{thm}
\begin{proof}
Let us first show that $\Phi_{[0,n)}^{\emptyset}(n)$ converges
in distribution to $\Upsilon_0$ when $n$ tends to $\infty$.
From Theorem \ref{thm:xx}, for all functions $h$ with bounded support $C$
(in particular continuous),
\begin{equation}
\label{eq:varto}
\sup_{A\in {\cal B}(\R)} \left|\P [\int h d \Phi_{[-n,0)}^{\emptyset}(0) \in A]
-\P [\int h d \Upsilon_0 \in A] \right| \le \P[n<\tau_0(C)]
\to 0
\end{equation}
when $n\to \infty$.
Since
$\Phi_{[0,n)}^{\emptyset}(n) =  \Phi_{[-n,0)}^{\emptyset}(0)\circ \theta_n$,
we can replace $\Phi_{[-n,0)}^{\emptyset}(0)$ by $\Phi_{[0,n)}^{\emptyset}(n)$
in (\ref{eq:varto}), which proves the convergence in distribution
of $\Phi_{[0,n)}^{\emptyset}(n)$ to $\Upsilon_0$.
By arguments similar to those of Theorem \ref{thm:xx},
one gets from Theorem \ref{thm:compactcoupling} that
for initial conditions ${\cal Z}_0$ as above, and for all $C$ compact,
$\Phi_{[-n,0)}^\emptyset$ and $\Phi_{[-n,0)}^{{\cal Z}_0}$
couple on $C$ for $n$ larger than a finite threshold.
This in turn implies that $\Phi_{[-n,0)}^{{\cal Z}_0}$
and $\Upsilon_0$ couple on $C$ for $n$ larger than a finite threshold.
By arguments similar to those above, this finally
implies that $\Phi_{[0,n)}^{{\cal Z}_0}$ converges
in distribution to $\Upsilon_0$.
\end{proof}

\section{Balance Equations for Moment Measures} \label{sec:mom}

The aim of this section is to establish a hierarchy of 
integral relations between the higher order factorial
moment measures 
of the steady state SBD process $\Upsilon=\Upsilon_0$ on $\R^d$ constructed in
the previous sections.
We will denote the factorial
moment density of order $k$ by  $\rho^{[k]}(x_1,\ldots,x_k)$.
Notice that $\Upsilon$ is motion invariant
(stationary and isotropic).
Hence $\rho^{[1]}(x)=\beta$ (the intensity of $\Upsilon$),
$$\rho^{[2]}(x,y)=\rho^{[2]}_{st}(x-y)
=\rho^{[2]}_{mi}(\|x-y\|)$$
and for all $k\ge 2$, 
$$\rho^{[k]}(x_1,\ldots,x_k)=
\rho^{[k]}_{st}(x_2-x_1,\ldots,x_k-x_1).$$

We first establish these balance equations.
For $k\ge 2$, the $k$-balance equation 
relates the $k$-th to the $k-1$-st and the $k+1$-st factorial
moment density. For $k=1$, it relates the first and the second
densities.

Then we show how to use these equations to get bounds and 
approximations.
\subsection{Balance Equations}
\label{ss:be}
For all $f$ positive, if $\Upsilon=\sum_n \delta_{X_n}$, we have
$$ \E^0_{\Upsilon} [\sum_{n\ne 0} f(\|X_n\|)]=
\frac 1 \beta 
 \int_{\R^d }
f(||x||) \rho^{[2]}_{mi}(\|x\|) dx=
\frac {d \nu_d} \beta 
 \int_{\R^+ }
f(r) \rho^{[2]}_{mi}(r) r^{d-1}dr,$$
with $\nu_d$ the volume of the unit ball in $\R^d$.

In steady state
the mean number of deaths in a Borel set
$C$ in the time interval $[0,\epsilon]$ is
$$ \beta |C| \epsilon \E^0_{\Upsilon} [\sum_{n\ne 0} f(\|X_n\|)] +o(\epsilon),$$
and it should be equal to the mean number of births in this set and time interval,
which is
$ \lambda |C| \epsilon.$
We get from this the following relation:
\begin{equation}
\label{eqo1}
\int_{\R^d} \rho^{[2]}_{mi}(\|x\|) f(||x||) dx=
\int_{\R^d} \rho^{[2]}_{st}(x) f(||x||) dx
= \lambda 
\end{equation}
which is our first balance relation which links 
the first and the second order
factorial moment densities.

Let $C_1, C_2$ be two Borel sets.
Let us look as above at the mean increase (due to births)
and the mean decrease (due to deaths) of the following quantity:
$$ \E \sum_{x_1\ne x_2\in \Upsilon} 1_{C_1}(x_1)1_{C_2}(x_2)= 
\int_{C_1}
\int_{C_2} \rho^{[2]}(x_1,x_2) dx_1 dx_2.$$
The mean increase in an interval of time of length $\epsilon$
is easily seen to be 
$$ \lambda \epsilon |C_1| \beta |C_2|
+\lambda \epsilon |C_2| \beta |C_1|+ o(\epsilon).$$
The mean decrease in the same interval is
\begin{eqnarray*}
&&\hspace{-.4cm} \E \sum_{x_1\ne x_2\in \Upsilon} 1_{C_1}(x_1)1_{C_2}(x_2) \epsilon
\left(
\sum_{z\in \Upsilon, z\ne x_1} f(\|z-x_1\|) 
+\sum_{z\in \Upsilon, z\ne x_2} f(\|z-x_2\|) \right)\\
& + & o(\epsilon)\\
&= & 2 \epsilon
\E \sum_{x_1\ne x_2\in \Upsilon} 1_{C_1}(x_1)1_{C_2}(x_2) f(\|x_1-x_2\|) \\
& + &
\epsilon \E \sum_{x_1, x_2, z \in \Upsilon, \mathrm{\ different}}
1_{C_1}(x_1)1_{C_2}(x_2) (f(\|z-x_1\|)+f(\|z-x_2\|)) + o(\epsilon)\\
&= &  2 \epsilon
\int_{C_1}\int_{C_2} f(\|x_1-x_2\|)  \rho^{[2]}(x_1,x_2) dx_1 dx_2\\
& +  &
\epsilon \int_{C_1}\int_{C_2} \int_{\R^d} f(\|z-x_1\|)+f(\|z-x_2\|)
\rho^{[3]}(x_1,x_2,z) dx_1 dx_2 dz +o(\epsilon).
\end{eqnarray*}
Hence
\begin{eqnarray}
\label{eqo2}
& & \hspace{-.8cm}2 \rho^{[2]}(x_1,x_2)  f(\|x_1-x_2\|)+
\int_{\R^d} \rho^{[3]}(x_1,x_2,z)
\left(f(\|x_1-z\|)+f(\|x_2-z\|) \right)
dz \nonumber \\
& = &  2\beta \lambda,
\end{eqnarray}
that is, for all $x\in \R^d$:
\begin{equation}
\label{eqo2p}
 2 \rho^{[2]}_{mi}(\|x\|)  f(\|x\|) +
\int_{\R^d} \rho^{[3]}_{st}(x,y)
\left(f(\|y\|)+f(\|y-x\|) \right)
dy = 2\beta \lambda,
\end{equation}
which is our second balance relation.

The general equation can be obtained in the same way.
Let us summarize our findings in:
\begin{thm}
\label{thm:baleq}
The factorial moment measures of the time stationary
SBD satisfy the following balance relations: 
\begin{equation}
\label{eqo1-gen}
\int_{\R^d} \rho^{[2]}_{mi}(\|x\|) f(||x||) dx=
 \lambda, 
\end{equation}
and for all $k\ge 2$, 
for all $x_1,\ldots,x_k$ in $\R^d$,
\begin{eqnarray}
& &  \rho^{[k]}(x_1,\ldots,x_k) \left(\sum_{i=1,k} \sum_{j=1,k, \ j\ne i}  
f(\|x_i-x_j\|)\right)
\nonumber \\ & &
 +
\int_{\R^d} \rho^{[k+1]}(x_1,\ldots,x_k,z)
\left(\sum_{i=1,k}
f(\|x_i-z\|)
\right) dz\nonumber \\
&&\qquad = \lambda \sum_{i=1,k} \rho^{[k-1]}(x_1,\ldots, x_{i-1},x_{i+1},\ldots,x_k) 
.
\label{gener}
\end{eqnarray}
\end{thm}
The general relation ($k\ge 2$) can be re-expressed in
terms of the functions $\rho^{[k]}_{st}$ as
\begin{eqnarray}
& &  \rho^{[k]}_{st}(x_2-x_1,\ldots,x_k-x_1)
\left(\sum_{i=1,k} \sum_{j=1,k, \ j\ne i}  
f(\|x_i-x_j\|)\right)
\nonumber \\ & &
 +
\int_{\R^d} \rho^{[k+1]}_{st}(x_1-z,\ldots,x_k-z)
\left(\sum_{i=1,k}
f(\|x_i-z\|)
\right) dz\nonumber \\
&&\qquad = \lambda \sum_{i=2,k} \rho^{[k-1]}_{st}(x_2-x_1,\ldots, x_{i-1}-x_1,x_{i+1},\ldots,x_k-x_1) \nonumber \\
&&\qquad\quad+ \lambda \rho^{[k-1]}_{st}(x_3-x_2,\ldots, x_{k}-x_2).
\label{generst}
\end{eqnarray}

\subsection{Bounds and Approximations}
\paragraph{Repulsion}
The next result says that in the stationary regime, there are less points
(in terms of their $f$--weight) around a typical point (i.e.\ under the
$\PP^0_\Upsilon$) than around a typical location of the
Euclidean plane (i.e.\ under the stationary probability $\PP$). Note that
this {\em $f$-repulsion} effect differs from what is usually called
repulsion (as in e.g. determinantal point processes).

\begin{thm}[$f$-repulsion]\label{thm:repulsivity} Under
the assumptions of Theorem \ref{thm:wholephi}, in the stationary regime,
\begin{equation}
\label{eq:palnopal}
\E\sum_{X\in \Upsilon} f(||X||) \ge \E^0_\Upsilon\sum_{X\in \Upsilon \setminus\{0\}} f(||X||).
\end{equation}
\end{thm}

\begin{proof}
From (\ref{eq:sumpressure-sher-simp}), in steady state
$$\lambda a \beta_\Upsilon=\beta_\Upsilon \E^0_\Upsilon(\pi_\Upsilon^2).$$
From (\ref{eq:difpsi'ante}), we also have
$$ \lambda= \beta_\Upsilon \E^0_\Upsilon(\pi_\Upsilon).$$
Hence
$$ a \beta_\Upsilon \E^0_\Upsilon(\pi_\Upsilon) = \E^0_\Upsilon(\pi_\Upsilon^2) \ge  \E^0_\Upsilon(\pi_\Upsilon)^2,$$
which directly gives (\ref{eq:palnopal}).
\end{proof}

\paragraph{First Order Approximation}
It follows from Theorem \ref{thm:repulsivity} that
\begin{equation}
\label{eqo1s}
\int_{\R^d} \rho^{[2]}_{st}(u) f(\|u\|) du \le
\beta^2 \int_{\R^d} f(\|u\|) du
=\beta^2 a.
\end{equation}
This and (\ref{eqo1}) give the bound:
\begin{equation}
\label{eqbound1}
\beta \ge \sqrt{\frac {\lambda}{a}}.
\end{equation}
This can actually be seen as an approximation of order 1 where
one (erroneously) pretends that $\rho^{[2]}_{st}(u)=\beta^2$.
The first order approximation of the intensity is hence
\begin{eqnarray}
\widehat \beta_1 = \sqrt{\frac {\lambda}{a}}
\end{eqnarray}
for the intensity and
\begin{equation}
\rho^{[k]}_1(x_1,\ldots,x^k)=\widehat \beta_1^k
\end{equation}
for the $k$-th moment measure.
\paragraph{Second Order Approximation}
Similarly, let us use the following approximation:
\begin{eqnarray*}
& & \int_{\R^d} \rho^{[3]}(x_1,x_2,z) (f(\|x_1-z\|)+f(\|x_2-z\|))dz\\
& \approx & \frac 1{\beta} \int_{\R^d} \rho^{[2]}(x_1,x_2) \rho^{[2]}(x_1,z)
f(\|x_1-z\|) dz\\ & + &
\frac 1{\beta} \int_{\R^d} \rho^{[2]}(x_1,x_2) \rho^{[2]}(x_2,z)
f(\|x_2-z\|) dz .
\end{eqnarray*}
We then get from this and from (\ref{eqo2})
the (heuristic) equation
\begin{eqnarray*}
\label{eqo2heur}
& & \hspace{-.8cm}2 \rho^{[2]}(x_1,x_2)  f(\|x_1-x_2\|)
\\ & + & \frac 1{\beta}
\rho^{[2]}(x_1,x_2)
\int_{\R^d} 
\left(
\rho^{[2]}(x_1,z) f(\|x_1-z\|))
+\rho^{[2]}(x_2,z) f(\|x_2-z\|)) \right)
dz \\
& \approx &  2\beta \lambda.
\end{eqnarray*}
This gives:
\begin{eqnarray}
\label{eqo2heurbis}
\rho^{[2]}_{mi}(r)
\approx \frac {\beta \lambda} {f(r) +  \mu },
\end{eqnarray}
with
$$ \mu = \frac 1\beta \int_{\R^d} 
\rho^{[2]}_{st}(x) f(\|x\|))
dx.$$ 
Multiplying (\ref{eqo2heurbis}) by $f(r)$ and integrating leads to the following
equation with unknown $\mu$:
\begin{eqnarray}
\label{eqo2heur1}
\mu \approx
d \nu_d  \lambda
\int_0^\infty  
\frac {f(r)} {f(r) +  \mu} r^{d-1} dr.
\end{eqnarray}
The left hand side of the last equation is increasing (in $\mu$)
from 0 to infinity whereas the right hand side is strictly decreasing 
from infinity to 0. Since both functions are continuous,
there is one and only one solution
to this equation that we will denote by $\widehat\mu_2$.
The second order approximation of the reduced second moment density 
$g(r)=  \frac 1 \beta \rho^{[2]}_{mi}(r)$
is then
\begin{eqnarray}
\label{eqo2heurreddens}
\widehat g_2(r)
= \frac {\lambda} {f(r) + \widehat \mu_2 },
\end{eqnarray}
whereas the second order approximation of the density is
\begin{eqnarray}
\label{eqo2heurintens}
\widehat \beta_2 =\lim_{r\to \infty} \widehat g_2(r)
= \frac {\lambda} {\widehat\mu_2 }.
\end{eqnarray}
and of course 
\begin{eqnarray}
\label{eqo2heursecmom}
\widehat \rho^{[2]}_{mi,2}(r)
= \frac {\lambda^2} {\widehat\mu_2(f(r) + \widehat \mu_2) }.
\end{eqnarray}
\paragraph{Third Order Approximation}
By the same arguments, the third order approximation is based 
on the equation
\begin{eqnarray}
\label{eqo3fp}
\rho^{[3]}(x_1,x_2,x_3)
\approx 
\frac{
\lambda \left( \rho^{[2]}(x_1,x_{2}) +\rho^{[2]}(x_2,x_{3}) +\rho^{[2]}(x_1,x_{3}) \right)}
{2 \left(
f(\|x_1-x_2\|)
+f(\|x_2-x_3\|)
+f(\|x_3-x_1\|)\right) + 3 \mu
},
\end{eqnarray}
with $\mu$ as defined above.
Let
\begin{eqnarray*}
h(x_1,x_2,z)& = &
\frac{f(\|x_1-z\|)
+f(\|x_2-z\|)}{2 \left(
f(\|x_1-x_2\|)
+f(\|x_1-z\|)
+f(\|x_2-z\|)\right) + 3 \mu
}\\
j(x_1,x_2) & = & 
\frac{1} {2 f(\|x_1-x_2\|) +
\lambda \int_{\R^d} h(x_1,x_2,z) dz}.
\end{eqnarray*}
Equations (\ref{eqo3fp}) and (\ref{eqo2})
lead to the following Volterra type integral 
equation for the third order approximation of
$g^{[2]}(x_1,x_2):=\frac 1\beta \rho^{[2]}(x_1,x_2)$:
\begin{eqnarray}
\label{eqvolt}
\widehat g^{[2]}_3 (x_1,x_2)
& = & 2\lambda j(x_1,x_2)\nonumber \\
& -  &  \lambda
j(x_1,x_2) 
\int_{\R^d}(\widehat g^{[2]}_3 (x_1,z) + \widehat g^{[2]}_3 (x_2,z))
h(x_1,x_2,z) dz .
\end{eqnarray}

\section{Appendix}\label{sec:appendix}

\subsection{Proof of Equation \eqref{eq:difpsi'ante}}
\label{appdiff}
Let $C$ be a bounded Borel set of Lebesgue measure 1.
Let $t$ be fixed and let $\epsilon>0$.
Given $\Phi'_{t}$, denote the points of
$\Phi'_{t}(C)$ by $X_1,\cdots,X_k$, and let $t_1,\ldots,t_k$
be independent exponential random variables with parameters
$\pi_{\Phi'_t}(X_1),\cdots,\pi_{\Phi'_t}(X_k)$, respectively.
Let $t_b$ be an independent exponential random variable
with parameter $\lambda$ and let $N_\epsilon$ be an independent Poisson
random variable with the parameter $\epsilon \lambda$. We have
\begin{eqnarray*}
\E [\Phi'_{t+\epsilon} (C) \mid \Phi'_{t}]
& \le & \Phi'_{t}(C) +
\P[N_\epsilon =1 \mid \Phi'_t ] \P[\min t_i > \epsilon \mid \Phi'_t] \\ & &
-\sum_{i=1}^k  
\P[N_\epsilon =0  \mid \Phi'_t] \P[t_i < \epsilon, \min_{j\ne i} t_j>\epsilon \mid \Phi'_t] \\ & &
+ \E[N_\epsilon \mid \Phi'_t]  \P[\mbox{ two or more of }
t_b,t_1,\ldots,t_k <\epsilon \mid \Phi'_t]\\
& \le & \Phi'_{t}(C) +
e^{-\lambda \epsilon} \lambda \epsilon -
\sum_{i=1}^k (1-e^{-\epsilon \pi_{\Phi'_{t}} (X_i)})
e^{ -\epsilon (\lambda +\sum_{j\ne i} \pi_{\Phi'_{t}} (X_j))}\\
& & + \lambda \epsilon \sum_{i=1}^k (1-e^{-\lambda \epsilon}) (1-e^{-\epsilon \pi_{\Phi'_{t}} (X_i)}) \\
& & + \lambda \epsilon \sum_{i\ne j =1}^k 
(1-e^{-\epsilon \pi_{\Phi'_{t}} (X_i)})
(1-e^{-\epsilon \pi_{\Phi'_{t}} (X_j)}). 
\end{eqnarray*}
Hence
\begin{eqnarray*}
\E [\Phi'_{t+\epsilon} (C) \mid \Phi'_{t}]
& \le & \Phi'_{t}(C) + \lambda \epsilon
\\ & & 
-\sum_{i=1}^k \epsilon \pi_{\Phi'_{t}} (X_i) 
\left(1- \frac 1 2 \epsilon  \pi_{\Phi'_{t}} (X_i)\right)
\left(1-\epsilon (\lambda + \sum_{j\ne i}  \pi_{\Phi'_{t}} (X_j))\right)\\
& & +
o(\epsilon) + \epsilon^2 k \sum_{i=1}^k  \pi_{\Phi'_{t}} (X_i),
\end{eqnarray*}
where $o(\epsilon)$ is deterministic.
Taking now expectation w.r.t.\ the point proccess $\Phi'_t$, we get
\begin{eqnarray*}
\beta_{\Phi'_{t+\epsilon}}
& \le & \beta_{\Phi'_{t}} +
\lambda \epsilon -\beta_{\Phi'_{t}} \epsilon \E^0_{\Phi'_{t}} \pi_{\Phi'_{t}} (0) \\
& & + \frac 12 \epsilon^2
\E \sum_{X\in \Phi'_{t}\cap C} \pi^2_{\Phi'_{t}} (X)
+ \epsilon^2 \lambda  \E \sum_{X\in \Phi'_{t}} \pi_{\Phi'_{t}} (X)\\ & & +
\epsilon^2 \beta_{\Phi'_{t}} \E \sum_{X\ne Y\in \Phi'_{t}\cap C} \pi_{\Phi'_{t}} (Y)\\ & & +
o(\epsilon) + \epsilon^2 \E \Phi'_t(C) \sum_{X\in\Phi'_t\cap C}  \pi_{\Phi'_{t}} (X).
\end{eqnarray*}
From Lemma \ref{lem:boundspecpres} (or more precisely
from its extension to $\Phi'$),
$\beta_{\Phi'_{t}} \E^0_{\Phi'_{t}} \pi_{\Phi'_{t}} (0)$
is finite for all $t$.
More generally, the proof of Lemma
\ref{lem:boundspecpres}
can easily be extended to show that each term of the last equation
involving an expectation is finite. This implies that
$$ \limsup_{\epsilon\to 0} \frac {\beta_{\Phi'_{t+\epsilon}} -\beta_{\Phi'_{t}}} {\epsilon} \le 
\lambda -\beta_{\Phi'_{t}} \E^0_{\Phi'_{t}} \pi_{\Phi'_{t}} (0).$$
The inferior limit is derived using similar techniques.
\subsection{Proof of Equations \eqref{eq:fac} and \eqref{eq:mtpua}}
\label{app:mtpua}
We first recall the general form of the mass transport principle.
Let $(\Omega,{\cal F}, \PP)$ be a probability space
endowed with a shift $\theta_u$, $u\in \R^d$.
Let $N$ and $N'$ be two $\theta_u$-compatible point processes on $\R^d$,
with respective intensities $\beta_N$ and $\beta_{N'}$ and Palm
probabilities $\PP_0$ and $\PP_0'$. Then, for all functions
$g:\R^d\times \Omega\to \R^+$, one has
$$ \beta_N \E_0 \int_{\R^d} g(y,\omega) N'(dy)
=\beta_{N'} \E_0' \int_{\R^d} g(-x,\theta_x(\omega)) N(dx).
$$

We now give the proof for \eqref{eq:mtpua}; \eqref{eq:fac} can be obtained exactly the same way.

The R.H.S. in  \eqref{eq:mtpua} can be rewritten as
$$ \beta_{{\cal R}_t} |C| \E^0_{{\cal R}_t} 
 \sum_{Y\in {\cal Z}_t} f(|Y|) \pi_{{\cal R+Z+A}_t} (0)
= |C| \beta_{{\cal R}_t} \E^0_{{\cal R}_t}
\int_{\R^d} g(y,\omega) {\cal Z}_t(dy),
$$
with $g(y,\omega)=f(|y|)\pi_{{\cal R+Z+A}_t}(0)$.
From the mass transport principle, 
\begin{eqnarray*}
\beta_{{\cal R}_t} \E^0_{{\cal R}_t}
\int_{\R^d} g(y,\omega) {\cal Z}_t(dy),
& = & 
\beta_{{\cal Z}_t} \E^0_{{\cal Z}_t}
\int_{\R^d} g(-x,\theta_x \omega) {\cal R}_t(dx)\\
& = & 
\beta_{{\cal Z}_t} \E^0_{{\cal Z}_t}
\sum_{X\in {\cal R}_t} f(|X|) \pi_{{\cal R+Z+A}_t\circ \theta_X}(0) \\
& = & 
\beta_{{\cal Z}_t} \E^0_{{\cal Z}_t}
\sum_{X\in {\cal R}_t} f(|X|) \pi_{{\cal R+Z+A}_t}(X) .
\end{eqnarray*}
The L.H.S. in  (\ref{eq:mtpua}) can be rewritten as
\begin{eqnarray*}
& & \hspace{-2cm}\E \sum_{Y\in {\cal Z}_t\cap C} \sum_{X\in {\cal R}_t}
f(|X-Y|) \pi_{{\cal R+Z+A}_t} (X)\\
& = & \E \sum_{Y\in {\cal Z}_t\cap C} \sum_{X\in {\cal R}_t}
f(|X-Y|) \pi_{{\cal R+Z+A}_t\circ \theta_Y} (X-Y)\\
& = & \beta_{{\cal Z}_t} |C| \E^0_{{\cal Z}_t}
\sum_{X\in {\cal R}_t} f(|X|) \pi_{{\cal R+Z+A}_t} (X),
\end{eqnarray*}
which concludes the proof.

\subsection{Table of Notation}
\label{sec:nota}

\begin{center}
\begin{longtable}{|m{3.9cm}|m{9.2cm}|}
\hline
$a=\int_{\Re^d}f(\|x\|)dx$ & Strength of the response function\\
\hline
$\mathcal{A}$& Set of antizombies\\
\hline
$\mathcal{A}(z)$& Set of antizombies offspring of $z$ \\
\hline
$b_p$& Birth time of point $p$ \\
\hline
$\beta_X$& Intensity of the stationary point process $X$ on $\R^d$\\
\hline
$ B(x,r) $ & Ball of radius $ r $ centered in $ x $ \\
\hline
$d$ &  Dimension of the Euclidean space \\
\hline
$D$ &  Convex set of $\R^d$ \\
\hline
$d_p$ & Death time of point $p$ \\ 
\hline
$\delta_x$ & Dirac measure at $x$ \\
\hline
$ \E^0_\chi $ & Palm probability of the $ \chi $ point process \\
\hline
$f:\R^+\to \R^+$ &  Response function \\
\hline
$I_{pq}$  & Connection direction for $(p,q)$\\
\hline
IS & Investigation stack\\
\hline
$K$ & Upper-bound on $f$ \\
\hline
$\lambda$& Birth rate\\
\hline
$l^d(C)=|C|$ & Lebesgue measure on $\R^d$ of the Borel set $ C $
\\
\hline
$\mu$ & Death rate \\ 
\hline
$ \nu_d $ & volume of a unit ball \\ 
\hline
$\Phi_t$ & Counting measure of the nodes living at time $t$ for the $\emptyset$ initial condition as obtained by Sheriff\\
\hline
$\Phi'_t$  & Counting measure of the nodes living at time $t$ for the $\cal Z$ initial condition as obtained by Sheriff\\
\hline
$\widetilde \Phi_t$ & Counting measure of the nodes unfinished at time $t$ as obtained by Sheriff$^Z$\\
\hline
$\pi_X(z)=\sum\limits_{x\in X} f(|x-z|)$ & Death pressure exerted on $z$ by the point process $X$.\\
\hline
$\Pi_X(Y)=\sum_{y\in Y} \pi_X(y)$ & Death pressure exerted on the point process $Y$ by the point process $X$\\
\hline
$p,q,\ldots$ &  Points of $\Psi$\\
\hline
$\Psi$ & Arrival counting measure on $\R\times \R^d$\\
\hline
$\Psi_{t}$ & Arrival counting measure on $\R\times \R^d$ after
time $t$ \\
\hline
$\Psi_{(s,t)}$ & Arrival counting measure on $\R\times \R^d$ in interval $(s,t)$\\
\hline
${\cal R}$ & Set of regular points\\ 
\hline
RCG & Random Connection Graph\\
\hline
SBD & Spatial Birth and Death\\
\hline
Sheriff & Pathwise construction of the set of nodes living at all $t$\\
\hline
Sheriff$^Z$ & Simultaneous construction of the set of nodes living at all $t$, for two different initial conditions\\ 
\hline
$S_p$ & Stack of node $p$\\
\hline
$\mathcal{S}$ & Set of special points
\\ \hline
$\mathcal{S}(z)$ & Set of special points offspring of $z$\\
\hline
$t$ & Time \\
\hline
$T_{pq}$ & Connection time for $(p,q)$\\
\hline
$x,y,\ldots$ & Points of $\Phi$\\
\hline
$x_p$ & Birth location of point $p$ \\
\hline
$\mathcal{Z}$ & Set of zombies
\\ \hline
$\mathcal{Z}_0$ & Initial condition point process \\
\hline
$\mathcal{Z}(z)$ & Set of zombies offspring of $z$
\\ \hline
$\mathfrak{z}(s)$ & Ancestor of the special node $s$ \\
\hline

\end{longtable}
\end{center}

\section*{Acknowledgements}

The authors thank Mayank Manjrekar for his comments on the paper.

The research of the first author was supported by an award from
the Simons Foundation (\# 197982 to The University of Texas at Austin). The work of the third author
was supported by the Academy of Finland with Senior Researcher funding and by EIT ITC Labs through the project DCDWN. 

The work presented in this paper has been partly carried out at
LINCS (\url{http://www.lincs.fr}).

\bibliographystyle{IEEEtran}

\label{LastPage}
\end{document}